\documentclass{amsart}

\usepackage[hypertex]{hyperref}

\usepackage{graphicx}
\usepackage{amsmath}
\usepackage{amsthm}
\usepackage{amsfonts}
\usepackage{amssymb, amscd}

\usepackage{euscript}
\usepackage[all,tips]{xy}
\usepackage{epic,eepic}

\usepackage{color}

\newtheorem{theorem}[subsection]{Theorem}

\newtheorem{proposition}[subsection]{Proposition}
\newtheorem{corollary}[subsection]{Corollary}

\theoremstyle{definition}
\newtheorem{definition}[subsection]{Definition}
\newtheorem{example}[subsection]{Example}
\newtheorem{remark}[subsection]{Remark}

\numberwithin{equation}{section}

\newcommand{\K}{\mathbb K}

\newcommand{\g}{\mathfrak{g}}

\newcommand{\sll}{\mathfrak{sl}_2(\mathbb{K})}

\newcommand{\relphantom[1]}{\mathrel{\phantom{#1}}}

\newcommand{\cat}[1]{{\EuScript #1}}

\newcommand{\cO}{\cat{O}}

\newcommand{\cF}{\cat{F}}

\newcommand{\bD}{\mathbf{D}}
\newcommand{\bE}{\mathbf{E}}

\newcommand{\bbK}{\mathbb{K}}
\newcommand{\bbZ}{\mathbb{Z}}

\newcommand{\HomMod}{\mathbf{HomMod}}
\newcommand{\Homdidend}{\mathbf{HomDend}}

\newcommand{\Homtridend}{\mathbf{HomTridend}}
\newcommand{\HomRB}{\mathbf{HomRB}}
\newcommand{\HomRBA}{\mathbf{HomRBA}}

\DeclareMathOperator{\Id}{Id}

\newcommand{\nicearrow}{\SelectTips{cm}{10}}
\newcommand{\bracket}{[\ ,\ ]}
\newcommand{\alphaass}{\mathfrak{as}_{A}}
\newcommand{\assal}{\mathfrak{as}_{A_l}}
\newcommand{\assprime}{\mathfrak{as}_{A'}}
\newcommand{\aoneass}{\mathfrak{as}_{A_1}}
\newcommand{\atwoass}{\mathfrak{as}_{A_2}}
\newcommand{\anmass}{\mathfrak{as}_{A(n,m)}}
\newcommand{\betaass}{\mathfrak{as}_{A_\beta}}
\newcommand{\Sthree}{\mathcal{S}_3}
\newcommand{\andspace}{\quad\text{and}\quad}
\newcommand{\Rtilde}{\widetilde{R}}

\newcommand{\tone}{{
\begin{picture}(2,10)(0,3)
\drawline(1,0)(1,10)
\put(1,10){\circle*{2}}
\end{picture}
}}

\newcommand{\ttwo}{{
\begin{picture}(16,14)(0,5)
\drawline(8,0)(8,6)(0,14)
\drawline(8,6)(16,14)
\put(0,14){\circle*{2}}
\put(16,14){\circle*{2}}
\put(8,6){\circle*{2}}
\end{picture}
}}

\newcommand{\tthreeone}{{
\begin{picture}(16,14)(0,5)
\drawline(8,0)(8,6)(0,14)
\drawline(8,6)(16,14)
\drawline(12,10)(8,14)
\put(0,14){\circle*{2}}
\put(16,14){\circle*{2}}
\put(8,14){\circle*{2}}
\put(8,6){\circle*{2}}
\put(12,10){\circle*{2}}
\end{picture}
}}

\newcommand{\tthreetwo}{{
\begin{picture}(16,14)(0,5)
\drawline(8,0)(8,6)(0,14)
\drawline(4,10)(8,14)
\drawline(8,6)(16,14)
\put(0,14){\circle*{2}}
\put(16,14){\circle*{2}}
\put(8,14){\circle*{2}}
\put(8,6){\circle*{2}}
\put(4,10){\circle*{2}}
\end{picture}
}}

\newcommand{\tfourone}{{
\begin{picture}(18,15)
\drawline(9,-4)(9,2)(0,11)
\drawline(9,2)(18,11)
\drawline(12,5)(6,11)
\drawline(15,8)(12,11)
\put(0,11){\circle*{2}}
\put(18,11){\circle*{2}}
\put(6,11){\circle*{2}}
\put(12,11){\circle*{2}}
\put(9,2){\circle*{2}}
\put(12,5){\circle*{2}}
\put(15,8){\circle*{2}}
\end{picture}
}}

\newcommand{\tfourtwo}{{
\begin{picture}(18,15)
\drawline(9,-4)(9,2)(0,11)
\drawline(9,2)(18,11)
\drawline(12,5)(6,11)
\drawline(9,8)(12,11)
\put(0,11){\circle*{2}}
\put(18,11){\circle*{2}}
\put(6,11){\circle*{2}}
\put(12,11){\circle*{2}}
\put(9,2){\circle*{2}}
\put(12,5){\circle*{2}}
\put(9,8){\circle*{2}}
\end{picture}
}}

\newcommand{\tfourthree}{{
\begin{picture}(18,15)
\drawline(9,-4)(9,2)(0,11)
\drawline(3,8)(6,11)
\drawline(9,2)(18,11)
\drawline(15,8)(12,11)
\put(0,11){\circle*{2}}
\put(18,11){\circle*{2}}
\put(6,11){\circle*{2}}
\put(12,11){\circle*{2}}
\put(9,2){\circle*{2}}
\put(3,8){\circle*{2}}
\put(15,8){\circle*{2}}
\end{picture}
}}

\newcommand{\tfourfour}{{
\begin{picture}(18,15)
\drawline(9,-4)(9,2)(0,11)
\drawline(6,5)(12,11)
\drawline(9,8)(6,11)
\drawline(9,2)(18,11)
\put(0,11){\circle*{2}}
\put(18,11){\circle*{2}}
\put(6,11){\circle*{2}}
\put(12,11){\circle*{2}}
\put(9,2){\circle*{2}}
\put(6,5){\circle*{2}}
\put(9,8){\circle*{2}}
\end{picture}
}}

\newcommand{\tfourfive}{{
\begin{picture}(18,15)
\drawline(9,-4)(9,2)(0,11)
\drawline(3,8)(6,11)
\drawline(6,5)(12,11)
\drawline(9,2)(18,11)
\put(0,11){\circle*{2}}
\put(18,11){\circle*{2}}
\put(6,11){\circle*{2}}
\put(12,11){\circle*{2}}
\put(9,2){\circle*{2}}
\put(6,5){\circle*{2}}
\put(3,8){\circle*{2}}
\end{picture}
}}

\newcommand{\psiwedgephi}{{
\begin{picture}(20,18)(-2,5)
\drawline(8,0)(8,6)(0,14)
\drawline(8,6)(16,14)
\put(8,6){\circle*{2}}
\put(-1,18){\makebox(0,0){$\psi$}}
\put(19,18){\makebox(0,0){$\varphi$}}
\end{picture}}}

\newcommand{\decthreetree}{{\setlength{\unitlength}{.14cm}
\begin{picture}(18,15)(-2,3)
\drawline(8,3)(8,6)(0,14)
\drawline(4,10)(8,14)
\drawline(8,6)(16,14)
\put(0,14){\circle*{.7}}
\put(16,14){\circle*{.7}}
\put(8,14){\circle*{.7}}
\put(8,6){\circle*{.7}}
\put(4,10){\circle*{.7}}
\put(-1,16){\makebox(0,0){\small$(3,5,2)$}}
\put(8,16){\makebox(0,0){\small$(0)$}}
\put(16,16){\makebox(0,0){\small$(2,6)$}}
\put(-1.5,9){\makebox(0,0){\small$(7,4,9,2)$}}
\put(12.5,5.5){\makebox(0,0){\small$(1,8,0)$}}
\end{picture}}}

\begin{document}

\title[ Rota-Baxter Hom-Lie-admissible algebras ]
{  Rota-Baxter Hom-Lie-admissible algebras }
\author{Abdenacer MAKHLOUF}
\author{  Donald YAU }
\address{Abdenacer Makhlouf, Universit\'{e} de Haute Alsace,  Laboratoire de Math\'{e}matiques, Informatique et Applications,
4, rue des Fr\`{e}res Lumi\`{e}re F-68093 Mulhouse, France}%
\email{Abdenacer.Makhlouf@uha.fr}

\address{Donald Yau, Department of Mathematics, The Ohio State University at Newark, 1179 University Drive, Newark, OH 43055, USA}
\email{dyau@math.ohio-state.edu}
\thanks {
}

%
\begin{abstract}
We study Hom-type analogs of Rota-Baxter and dendriform algebras, called Rota-Baxter $G$-Hom-associative algebras and Hom-dendriform algebras.  Several construction results are proved.  Free algebras for these objects are explicitly constructed.  Various functors between these categories, as well as an adjunction between the categories of Rota-Baxter Hom-associative algebras and of Hom-(tri)dendriform algebras, are constructed.
\end{abstract}
\maketitle

\section{Introduction}

Rota-Baxter operators have appeared in a wide range of areas in pure and applied mathematics. The paradigmatic example of a Rota-Baxter operator concerns the integration by parts formula. The algebraic formulation of a Rota-Baxter algebra  first appeared in G. Baxter's work in probability and the study of fluctuation theory \cite{Baxter}. This algebra was intensively studied by G.C. Rota \cite{Rota,Rota2} in connection with combinatorics. In the work of A. Connes and D. Kreimer \cite{ConnesKreimer} about their Hopf algebra approach to renormalization  of quantum field theory, the Rota-Baxter identity appeared  under the name \emph{multiplicativity constraint}. This seminal work gives rise to an important development including Rota-Baxter algebras and their connections to other algebraic structure.  See \cite{aguiar1,Aguiar3,KEF1,KEF-Guo1,KEF-Guo_2007, KEF-Guo-Kreimer,KEF-Guo-Manchon2006,KEF-Manchon_JA2009,KEF-Manchon2009,KEF-Bondia-Patras,KEF-Machon-Patras,Guo1,Guo2,Guo3,GuoK,Guo4,Li-Hou-Bai07,BaiBellierGuo}.
Rota-Baxter operator  in the context of Lie algebras were introduced
independently by Belavin and Drinfeld \cite{Drinfeld} and Semenov-Tian-Shansky \cite{semenov}.  Rota-Baxter Lie algebras are related to solutions of the (modified) classical Yang-Baxter equation.

Closely related to Rota-Baxter algebras are dendriform algebras, which were introduced by Loday in \cite{Loday1}. Dendriform algebras have two binary operations, which dichotomize an associative multiplication.  The motivation to introduce these algebraic structures comes from $K$-theory.  Dendriform algebras are connected to several areas in mathematics and physics, including  Hopf algebras, homotopy Gerstenhaber algebra, operads, homology, combinatorics, and quantum field theory, where they  occur  in the theory of renormalization of Connes and Kreimer.  Rota-Baxter algebras are related to dendriform algebras via a pair of adjoint functors \cite{KEF1,KEF-Guo2}.  Roughly speaking, Rota-Baxter algebras are to dendriform algebras as associative algebras are to Lie algebras.

The study of nonassociative algebras was originally motivated by certain problems in physics and other branches of mathematics. Hom-algebra structures first arose in quasi-deformations of Lie algebras of vector fields.
Discrete modifications of vector fields via twisted derivations lead to Hom-Lie and quasi-Hom-Lie structures, in which the Jacobi condition is twisted. The first examples of $q$-deformations, in which the derivations are replaced by $\sigma$-derivations, concerned the Witt and the Virasoro algebras; see for example \cite{AizawaSaito,ChaiElinPop,ChaiKuLukPopPresn,ChaiIsKuLuk,ChaiPopPres,CurtrZachos1,
DaskaloyannisGendefVir,Kassel1, LiuKeQin,Hu}. A  general study and construction of Hom-Lie algebras are considered in
\cite{HLS,LS1,LS2}.  A more general framework bordering color and super Lie algebras was introduced in \cite{HLS,LS1,LS2,LS3}. In the subclass of Hom-Lie
algebras, skew-symmetry is untwisted, whereas the Jacobi identity is twisted by a single linear map and contains three terms as in Lie algebras, reducing to ordinary Lie algebras when the twisting linear map is the identity map.

Hom-associative algebras, which generalize associative
algebras to a situation where associativity is twisted by a
linear map, was introduced  in \cite{MS}.  The
commutator bracket defined using the multiplication
in a Hom-associative algebra leads naturally to a Hom-Lie algebra.  This provides a different  way of constructing Hom-Lie algebras. Also introduced in \cite{MS} are Hom-Lie-admissible algebras and more general $G$-Hom-associative
algebras, which are twisted generalizations of Lie-admissible and $G$-associative algebras \cite{GozeRemm}, respectively.  The class of $G$-Hom-associative algebras includes the subclasses of Hom-Vinberg and Hom-preLie algebras,
generalizing to the twisted situation Vinberg and preLie algebras, respectively.  It was shown in \cite{MS} that for these classes of algebras the operation of taking
commutator leads to Hom-Lie algebras as well.

Enveloping
algebras of Hom-Lie algebras were discussed in \cite{Yau:EnvLieAlg,Yau:comodule}.
The fundamentals of the formal deformation theory and associated
cohomology structures for Hom-Lie algebras have been considered
initially  in \cite{HomDeform} and completed in \cite{AEM}. Homology for Hom-Lie algebras was developed in \cite{Yau:homology}.  In \cite{HomHopf} and \cite{HomAlgHomCoalg}, the theory of
Hom-coalgebras and related structures are developed. Further development could be found in, for example,  \cite{Mak-ElHamd:defoHom-alter,Mak:HomAlterna2010,Mak:Almeria,AmmarMakhloufJA2010,AM2008,Mak:HomAlterna2010,Canepl2009,JinLi},\cite{Yau:YangBaxter}-\cite{Yau:MalsevAlternJordan}, and the references therein.


The purpose of this paper is  to study a common generalization of Rota-Baxter algebras and Hom-Lie-admissible algebras, called Rota-Baxter Hom-Lie-admissible algebras, and the closely related Hom-dendriform algebras.   We explore their free algebras and the connections between their categories.  As we will show in this paper, Rota-Baxter operators provide new ways of going between the various subclasses of Hom-Lie-admissible algebras.


The rest of this paper is organized as follows.  We summarize in the next section the basics of Hom-Lie-admissible algebras.  In Section \ref{sect3}, we introduce Rota-Baxter $G$-Hom-associative algebras and provide several construction results.  In Section \ref{sec:rbtoprelie}, it is shown that some Rota-Baxter Hom-Lie-admissible and Hom-associative algebras yield left Hom-preLie algebras.  Free Rota-Baxter $G$-Hom-associative algebras are discussed in Section \ref{sec:HRB1}. The construction of the free Rota-Baxter $G$-Hom-associative algebra involves the combinatorial objects of decorated trees and is formally similar to the construction of the enveloping Hom-associative algebras of Hom-Lie algebras \cite{Yau:homology,Yau:comodule}.  Free Rota-Baxter algebras have been studied in \cite{am,KEF-Guo1}.

In Section \ref{sect4}, we discuss Hom-dendriform algebras, their associated Hom-associative and Hom-preLie algebras, and free Hom-dendriform algebras. The last section is dedicated to establishing functors between the categories of Rota-Baxter Hom-associative algebras and of Hom-(tri)dendriform algebras.

\section{ Hom-Lie admissible and  $G$-Hom-associative algebras } \label{sect1}

In this Section  we recall the main result of Hom-Lie-admissible algebras in \cite{MS} and summarize the definitions and some properties of $G$-Hom-associative algebras.  The latter are twisted generalizations of the  $G$-associative algebras introduced in \cite{GozeRemm}.

\subsection{Convention}
Throughout this article we work over a fixed field $\mathbb{K}$ of characteristic $0$.

\subsection{Hom-algebras}

A \emph{Hom-module} is a pair $(A,\alpha )$ consisting of a $\K$-module  $A$ and  a linear self-map $\alpha \colon A \rightarrow A$, called the \emph{twisting map}.  A \emph{morphism} of Hom-modules $f \colon (A,\alpha) \to (A',\alpha')$ is a linear map $f \colon A \to A'$ such that $f \circ \alpha = \alpha' \circ f$.  The category of Hom-modules is denoted by $\HomMod$.

A \emph{Hom-algebra} is  a triple $(A,\mu,\alpha )$ consisting of a Hom-module  $(A,\alpha )$ and a  bilinear map $\mu \colon A\times A \rightarrow A$.  Such a Hom-algebra is often denoted simply by $A$.  A Hom-algebra $(A, \mu, \alpha)$  is said to be
\emph{multiplicative} if for all  $ x,y\in A$ we have $\alpha(\mu(x,y))=\mu(\alpha (x),\alpha (y))$.

Let $\left( A,\mu,\alpha \right) $ and
$A^{\prime }=\left( A^{\prime },\mu ^{\prime
},\alpha^{\prime }\right) $ be two Hom-algebras.  A \emph{morphism} $f \colon A\rightarrow A^{\prime }$ of Hom-algebras is a linear map such that
$
\mu ^{\prime }\circ (f\otimes f)=f\circ \mu \ $
and $ \  f\circ \alpha=\alpha^{\prime }\circ f.
$

In particular, two Hom-algebras $\left( A,\mu,\alpha \right) $ and
$\left( A,\mu ^{\prime },\alpha^{\prime }\right) $ are \emph{isomorphic} if
there exists a
linear isomorphism $f\ $such that
$
\mu =f^{-1}\circ \mu ^{\prime }\circ (f\otimes f)$
and $\alpha= f^{-1}\circ \alpha^{\prime }\circ
f.
$


\subsection{Hom-Lie algebras}
The  notion of a Hom-Lie algebra  was introduced by Hartwig, Larsson and
Silvestrov in \cite{HLS,LS1,LS2}, motivated initially by examples of deformed
Lie algebras coming from twisted discretizations of vector fields.
In this article, we follow the notations and a slightly more general definition
of a Hom-Lie algebra from \cite{MS}.

\begin{definition}\label{def:HomLie}
A \emph{Hom-Lie algebra} is a Hom-algebra $(\g, [\ ,
\ ], \alpha)$ that satisfies
\begin{eqnarray} & [x_1,x_2]=-[x_2,x_1],
\quad {\text{(skew-symmetry)}} \\ \label{HomJacobiCondition} &
\circlearrowleft_{x_1,x_2,x_3}{[\alpha(x_1),[x_2,x_3]]}=0 \quad
{\text{(Hom-Jacobi identity)}}
\end{eqnarray}
for all $x_1, x_2, x_3$ in $\g$, where $\circlearrowleft_{x_1, x_2, x_3}$ denotes the summation over the cyclic permutations of $x_1, x_2, x_3$.
\end{definition}

When the twisting map $\alpha$ is the identity map, we recover classical Lie algebras.  The Hom-Jacobi identity \eqref{HomJacobiCondition} is the Jacobi identity in this case.

The classical concept of Lie-admissible algebras (see for example  \cite{GozeRemm}) is extended to the Hom-setting in \cite{MS} as follows:

\begin{definition}
A \emph{Hom-Lie-admissible algebra} is a Hom-algebra $(A,\mu,\alpha )$ in which the commutator bracket, defined for all $x_1,x_2 \in A$ by
$$
[ x_1,x_2 ]=\mu (x_1,x_2)-\mu (x_2,x_1 ),
$$
satisfies the Hom-Jacobi identity \eqref{HomJacobiCondition}.
\end{definition}

\begin{remark}
\label{commutatorhomlie}
The commutator bracket is automatically skew-symmetric.  Thus, if it satisfies the Hom-Jacobi identity, then it defines a Hom-Lie algebra $(A, \bracket, \alpha)$.  When $(A,\mu,\alpha)$ is a Hom-Lie-admissible algebra, the Hom-Lie algebra $(A, \bracket, \alpha)$ is also called the \emph{commutator Hom-Lie algebra}.
\end{remark}

Let $(A,\mu,\alpha )$ be a Hom-algebra.  The \emph{Hom-associator} of $A$ is the trilinear map
$\alphaass$ on $A$ defined  for $x_1,x_2,x_3\in A$ by
\begin{equation}\label{Hom-associator}
\alphaass(x_1,x_2,x_3)=\mu (\mu (x_1,x_2),\alpha
(x_{3}))-\mu(\alpha(x_{1}),\mu (x_{2},x_3)).
\end{equation}
When the twisting map $\alpha$ is the identity map, the Hom-associator reduces to the usual associator.

\begin{definition}
Let $G$ be a subgroup of the permutation group $\mathcal{S}_3$.  A \emph{$G$-Hom-associative algebra} is a  Hom-algebra $(A,\mu,\alpha )$ that satisfies the identity
\begin{equation}\label{admi}
\sum_{\sigma\in G} (-1)^{\varepsilon ({\sigma})} \alphaass\left(x_{\sigma
(1)}, x_{\sigma (2)}, x_{\sigma (3)}\right) = 0,
\end{equation}
where the $x_i$ are in $A$, and $(-1)^{\varepsilon ({\sigma})}$ is the
signature of the permutation $\sigma$.
\end{definition}

The condition (\ref{admi}) is called the \emph{$G$-Hom-associative identity}.  It is equivalent to
\[
\sum_{\sigma\in G}{(-1)^{\varepsilon ({\sigma})} \alphaass \circ
\sigma}=0,
\]
where
$\sigma(x_1, x_2, x_3) = (x_{\sigma(1)}, x_{\sigma(2)}, x_{\sigma(3)})$.  When $\alpha = \Id$, \eqref{admi} is called the \emph{$G$-associative identity} \cite{GozeRemm}.

\begin{remark}
Suppose $(A,\mu,\alpha)$ is  a Hom-algebra, and
$
[ x,y ]=\mu (x,y)-\mu (y,x )
$ denotes the commutator bracket.  Then the Hom-algebra $(A, \bracket, \alpha)$  satisfies  the Hom-Jacobi identity if and only if the $\Sthree$-Hom-associative identity \eqref{admi}
\[
\sum_{\sigma\in \Sthree} (-1)^{\varepsilon ({\sigma})} \alphaass\left(x_{\sigma
(1)}, x_{\sigma (2)}, x_{\sigma (3)}\right) = 0
\]
holds for all $x_i \in A$.  In other words, Hom-Lie-admissible algebras are all $\Sthree$-Hom-associative algebras.  The following result gives the converse of this observation.
\end{remark}

\begin{proposition}[\cite{MS}]\label{ThmAdmi}
Let $G$ be a subgroup of the permutation group $\mathcal{S}_3$.  Then every $G$-Hom-associative algebra is a Hom-Lie-admissible algebra.
\end{proposition}

The subgroups of $\Sthree$ are
\[
\begin{split}
G_1&=\{\Id\}, ~G_2=\{\Id,\sigma_{1 2}\},~G_3=\{\Id,\sigma_{2
3}\},\\
G_4&=\{\Id,\sigma_{1 3}\},~G_5=A_3 ,~G_6=\mathcal{S}_3,
\end{split}
\]
where $A_3$ is the alternating group, and $\sigma_{ij}$ is the transposition between $i$ and $j$.  In view of Proposition \ref{ThmAdmi}, it is natural to introduce the following subclasses of Hom-Lie-admissible algebras corresponding to the subgroups of $\mathcal{S}_3.$
\begin{itemize}
\item
The  $G_1$-Hom-associative algebras are also called \emph{Hom-associative
algebras} and satisfy the $G_1$-Hom-associative identity
\begin{equation}\label{HomAssCond}
\mu(\alpha(x_1),\mu (x_2, x_3))= \mu(\mu(x_1,x_2), \alpha ( x_3)),
\end{equation}
or equivalently
\[
\alphaass = 0.
\]
When the twisting map $\alpha$ is the identity map, we recover an associative algebra.

\item
The  $G_2$-Hom-associative algebras are also called \emph{left Hom-preLie  algebras}, \emph{Hom-Vinberg algebras}, and \emph{left Hom-symmetric algebras}.  They satisfy the $G_2$-Hom-associative identity
\begin{equation}\label{HomPreLieCond}
\mu(\alpha(x_1),\mu (x_2, x_3))-\mu(\alpha(x_2),\mu (x_1, x_3))=\mu (\mu
(x_1,x_2),\alpha ( x_3))-\mu (\mu (x_2,x_1),\alpha ( x_3)),
\end{equation}
which is equivalent to saying that the Hom-associator is symmetric in the first two variables.  When the twisting map $\alpha$ is the identity map, we recover a left preLie algebra, also known as a Vinberg algebra and a left symmetric algebra.

\item
The  $G_3$-Hom-associative algebras are also called \emph{right Hom-preLie  algebras} and \emph{right Hom-symmetric algebras}.  They satisfy the $G_3$-Hom-associative identity
\[
\mu(\alpha(x_1),\mu (x_2, x_3))-\mu(\alpha(x_1),\mu (x_3,x_2))=\mu (\mu
(x_1,x_2),\alpha ( x_3))-\mu (\mu (x_1, x_3),\alpha ( x_2)),
\]
which is equivalent to saying that the Hom-associator is symmetric in the last two variables.  When the twisting map $\alpha$ is the identity map, we recover a right pre-Lie algebra, also known as a right symmetric algebra.

\item
The  $G_4$-Hom-associative algebras satisfy the $G_4$-Hom-associative identity
\begin{equation}\nonumber
\mu(\alpha(x_1),\mu (x_2, x_3)) - \mu(\alpha(x_3),\mu (x_2,x_1))
=
\mu (\mu
(x_1,x_2),\alpha ( x_3))
- \mu (\mu ( x_3,x_2),\alpha (x_1)).
\end{equation}

\item
The  $G_5$-Hom-associative algebras satisfy the $A_3$-Hom-associative identity
\begin{eqnarray}\nonumber
\nonumber \mu(\alpha(x_1),\mu (x_2, x_3))+\mu(\alpha(x_2),\mu
( x_3,x_1))+\mu(\alpha( x_3),\mu
(x_1,x_2))= \\
\nonumber \mu (\mu (x_1,x_2),\alpha ( x_3))+\mu (\mu (x_2, x_3),\alpha (x_1))+\mu
(\mu ( x_3,x_1),\alpha (x_2)).
\end{eqnarray}\nonumber
If the product $\mu$ is skew-symmetric, then the previous condition is
exactly the Hom-Jacobi identity.

\item
The  $G_6$-Hom-associative algebras are exactly the Hom-Lie-admissible
algebras.
\end{itemize}

\begin{remark}
A Hom-associative algebra is also a $G$-Hom-associative algebra for every subgroup $G$ of $\Sthree$.  A left Hom-preLie algebra is the opposite algebra of a right Hom-preLie algebra.
\end{remark}


\section{Rota-Baxter operators and $G$-Hom-associative algebras } \label{sect3}
In this section, we extend the notion of Rota-Baxter algebra to $G$-Hom-associative  algebras and prove some construction results.

\begin{definition}
A \emph{Rota-Baxter  $G$-Hom-associative algebra (of weight $\lambda$)} is a $G$-Hom-associative algebra $( A, \cdot, \alpha) $ together with a linear self-map $R\colon A\rightarrow A$ that satisfies the identity
\begin{equation}\label{eq:RB}
R(x)\cdot R(y)=R(R(x)\cdot y+x\cdot R(y)+\lambda x\cdot y),
\end{equation}
where $\lambda\in\K$ is a fixed scalar called the \emph{weight}.  The map $R$ is called a \emph{Rota-Baxter operator}, and the identity \eqref{eq:RB} is called the \emph{Rota-Baxter identity}.  A \emph{morphism} of Rota-Baxter $G$-Hom-associative algebras is a morphism of Hom-algebras that commutes with the Rota-Baxter operators.  With $G$ understood, the category of Rota-Baxter $G$-Hom-associative algebras (of weight $\lambda$) is denoted by $\HomRB_\lambda$. A Rota-Baxter  $G$-Hom-associative algebra is \emph{multiplicative} if the Hom-algebra $(A, \cdot, \alpha)$ is multiplicative and $\alpha R = R\alpha$.
\end{definition}

When the weight $\lambda$ is understood, we denote a Rota-Baxter  $G$-Hom-associative algebra by a quadruple $( A, \cdot, \alpha,R). $   We obtain \emph{Rota-Baxter $G$-associative algebras} when the twisting map $\alpha$ is the identity map, and we denote them by triples $( A, \cdot,R)$.  A Rota-Baxter $G_1$-Hom-associative (resp., $G_6$-Hom-associative) algebra will also be called a \emph{Rota-Baxter Hom-associative} (resp., \emph{Hom-Lie-admissible}) \emph{algebra}.

\begin{remark}
A Rota-Baxter $\{\Id\}$-associative algebra is exactly what is usually called a Rota-Baxter algebra.  A Rota-Baxter $A_3$-associative algebra whose multiplication is skew-symmetric is exactly a Rota-Baxter Lie algebra.
\end{remark}

\begin{remark}
A Rota-Baxter Hom-Lie algebra is a Rota-Baxter $A_3$-Hom-associative algebra in which the product $\mu$ is skew-symmetric.  Rota-Baxter Hom-Lie algebras are closely related to Hom-Novikov algebras; see \cite[Theorem 1.4]{Yau:novikov}.
\end{remark}


\begin{example}
Let $( A, \cdot, \alpha,R)$ be a  Rota-Baxter  $G$-Hom-associative algebra of weight $\lambda$. Then
\[
( A, \cdot, \alpha,- R -\lambda \Id)
\]
is a  Rota-Baxter $G$-Hom-associative  algebra of weight $\lambda$.
\end{example}

\begin{example}
Let $( A, \cdot, \alpha,R)$ be a  Rota-Baxter  $G$-Hom-associative algebra of weight $\lambda$. Then
$
( A, \cdot, \alpha, -R )
$
is a  Rota-Baxter $G$-Hom-associative  algebra of weight $-\lambda$.
\end{example}

\begin{example}
Suppose $\lambda \in \K$ is a non-zero scalar.   Then $( A, \cdot, \alpha,R)$ is a  Rota-Baxter  $G$-Hom-associative algebra of weight $\lambda$ if and only if
\[
(A, \cdot, \alpha, \lambda^{-1} R)
\]
is a Rota-Baxter  $G$-Hom-associative algebra of weight $1$.
\end{example}

\begin{example}[Jackson $\mathfrak{sl}_2$]
The Jackson $\mathfrak{sl}_2$, denoted  $q\mathfrak{sl}_2$, is a $q$-deformation of the classical  $\mathfrak{sl}_2$. This family of Hom-Lie algebras was constructed in \cite{LS3}
using  a quasi-deformation scheme based on discretizing by means of
Jackson $q$-derivations a representation of $\sll$ by
one-dimensional vector fields (first order ordinary differential
operators) and using the twisted commutator bracket defined in
\cite{HLS}. It carries a  Hom-Lie algebra structure but not a Lie algebra structure for $q\neq 1$. It is defined with respect to a basis $\{x_1,x_2,x_3\}$  by the brackets and a linear  map $\alpha$ such that:
$$
\begin{array}{cc}
\begin{array}{ccc}
 [ x_1, x_2 ] &= &-2q x_2, \\ {}
 [x_1, x_3 ]&=& 2 x_3, \\ {}
 [ x_2,x_3 ] & = & - \frac{1}{2}(1+q)x_1,
 \end{array}
 & \quad
  \begin{array}{ccc}
  \alpha (x_1)&=&q x_1, \\
 \alpha (x_2)&=&q^2  x_2, \\
   \alpha (x_3)&=&q x_3,
  \end{array}
\end{array}
$$
where $q$ is a parameter in $\K$. If $q=1$ we recover the classical $\mathfrak{sl}_2$.

Let $R_1$ and $R_2$ be the following two operators defined on $q\mathfrak{sl}_2$ with respect to the basis $\{x_1,x_2,x_3\}$ by:
$$
\begin{array}{cc}
\begin{array}{ccc}
 R_1(x_1)&=&  0,\\ {}
R_1( x_2 )&=&\rho_1 x_2 +q  \frac{\rho_1^2}{\rho_2}  x_3, \\ {}
R_1(x_3 ) & =&  \rho_2 x_2 +q  \rho_1  x_3,
 \end{array}
 & \quad
  \begin{array}{ccc}
   R_2(x_1)&=&  \rho_1 x_1 - \frac{(1+q) \rho_1^2}{\rho_2}  x_3,\\ {}
R_2( x_2 )&=&\frac{1}{2}\rho_1 x_2 +\frac{(1+q)^2 \rho_1^3}{8\rho_2^2}  x_3, \\ {}
R_3(x_3 ) & =& \rho_2 x_1 +\frac{ 2q \rho_2^2}{(1+q)^2\rho_1}  x_2-\frac{(2+q) \rho_1}{2}  x_3,
  \end{array}
\end{array}
$$
where $\rho_1,\rho_2$ are  parameters in $\K^*$ and $q\neq -1$.  Then $(q\mathfrak{sl}_2,[\ , \ ],\alpha, R_1)$ and $(q\mathfrak{sl}_2,[\ , \ ],\alpha, R_2)$ are Rota-Baxter Hom-Lie algebras of weight $0$.
\end{example}

\begin{example}\label{example1ass}
Let $\{x_1,x_2,x_3\}$  be a basis of a $3$-dimensional vector space
$A$ over $\K$. The following multiplication $\cdot$ and linear map
$\alpha$ on $A$ define Hom-associative algebras over $\K^3${\rm :}
$$
\begin{array}{ll}
\begin{array}{lll}
  x_1\cdot x_1&=& a\ x_1, \ \\
 x_1\cdot  x_2&=& x_2 \cdot  x_1=a\ x_2,\\
x_1 \cdot  x_3 &=& x_3\cdot  x_1=b\ x_3,\\
 \end{array}
 & \quad
 \begin{array}{lll}
x_2\cdot  x_2 &=& a\ x_2, \ \\
 x_2\cdot   x_3&=& b\ x_3, \ \\
 x_3\cdot  x_2&=&  x_3\cdot x_3=0,
  \end{array}
\end{array}
$$

$$  \alpha (x_1)= a\ x_1, \quad
 \alpha (x_2) =a\ x_2 , \quad
   \alpha (x_3)=b\ x_3,
$$
where $a,b$ are parameters in $\K$. The algebras are not associative
when $a\neq b$ and $b\neq 0$, since
$$ (x_1\cdot  x_1)\cdot x_3-  x_1\cdot
(x_1\cdot x_3)=(a-b)b x_3.$$

Let $R$  be the  operator defined with respect to the basis $\{x_1,x_2,x_3\}$ by
$$
 R(x_1)=  \rho_1 x_3,\
R( x_2 )=\rho_2 x_3, \
R(x_3 )  =0,
$$
where $\rho_1,\rho_2$ are  parameters in $\K$.  Then $(A,\cdot,\alpha, R)$ is a Rota-Baxter Hom-associative algebra of weight $0$.

\end{example}

In the rest of this Section, we apply the twisting principles, introduced in \cite[Theorem 2.4]{Yau:homology} and \cite[Theorem 2.11]{Yau:MalsevAlternJordan},  to obtain examples of Rota-Baxter $G$-Hom-associative algebras.  Also, we extend to Rota-Baxter $G$-associative algebras the  construction involving elements of the centroid used in \cite[Proposition 1.12]{BenayadiMakhlouf}.

The following result states that a Rota-Baxter  $G$-Hom-associative algebra yields another Rota-Baxter $G$-Hom-associative algebra when its multiplication and twisting map are twisted by a morphism.

\begin{theorem}
\label{thm:firsttwist}
Let $(A,\mu, \alpha, R)$ be a Rota-Baxter $G$-Hom-associative algebra of weight $\lambda$ and  $\beta \colon A\rightarrow A$ be a morphism.  Then
\[
A_\beta = (A, \mu_\beta = \beta\mu, \beta\alpha, R)
\]
is also a Rota-Baxter $G$-Hom-associative algebra of weight $\lambda$.  Moreover, if $A$ is multiplicative, then so is $A_\beta$.

Furthermore, suppose $(A',\mu', \alpha', R')$ is a Rota-Baxter  $G$-Hom-associative
algebra   and $\beta ' \colon A'\to A'$ is a morphism.  If $f\colon A\rightarrow A'$ is a morphism such that $f\beta = \beta'f$, then
$$f\colon A_\beta \to A'_{\beta'}
$$
is also a morphism.
\end{theorem}



\begin{proof}
To see that $A_\beta$ is a $G$-Hom-associative algebra, note that the Hom-associators of $A$ and $A_\beta$ are related as
\[
\betaass = \beta^2 \circ \alphaass.
\]
Thus, the $G$-Hom-associative identity in $A$ implies
\[
\sum_{\sigma\in G} (-1)^{\varepsilon ({\sigma})} \betaass \circ \sigma
= \beta^2 \circ \left(\sum_{\sigma\in G} (-1)^{\varepsilon ({\sigma})} \alphaass \circ
\sigma\right) = 0.
\]
This shows that $A_\beta$ is a $G$-Hom-associative algebra.

The Rota-Baxter identity \eqref{eq:RB} for $A_\beta$ is true by the following commutative diagram:
\begin{equation}
\label{muonerb}
\nicearrow
\xymatrix{
A^{\otimes 2} \ar[r]^-{\chi} \ar[dd]_-{R^{\otimes 2}} & A^{\otimes 2} \ar[r]^-{\mu} & A \ar[r]^-{\beta} \ar[dd]_-{R} & A \ar[dd]^-{R}\\
& & & \\
A^{\otimes 2} \ar[rr]^-{\mu} & & A \ar[r]^-{\beta} & A.
}
\end{equation}
Here
\[
\chi = R \otimes \Id + \Id \otimes R + \lambda (\Id^{\otimes 2}),
\]
so the left square is commutative by the Rota-Baxter identity in $A$.  The right rectangle is commutative because $\beta$ is a morphism.  We have shown that $A_\beta$ is a Rota-Baxter $G$-Hom-associative algebra.

The assertions concerning multiplicativity and the morphism $f$ are obvious.
\end{proof}

Let us discuss two special cases of Theorem \ref{thm:firsttwist}.  The following result says that each multiplicative Rota-Baxter $G$-Hom-associative algebra yields a sequence of multiplicative Rota-Baxter $G$-Hom-associative algebras with twisted multiplications and twisting maps.

\begin{corollary}
Let $(A,\mu, \alpha, R)$ be a multiplicative Rota-Baxter $G$-Hom-associative algebra of weight $\lambda$.  Then
\[
A_{\alpha^n} = \left(A, \alpha^n\mu, \alpha^{n+1}, R\right)
\]
is also a multiplicative Rota-Baxter $G$-Hom-associative algebra of weight $\lambda$ for each $n \geq 1$.
\end{corollary}

\begin{proof}
This is the $\beta = \alpha^n$ special case of Theorem \ref{thm:firsttwist}.
\end{proof}

The following result says that a Rota-Baxter $G$-associative algebra deforms into a multiplicative Rota-Baxter $G$-Hom-associative algebra via a morphism.

\begin{corollary}
\label{cor1:twist}
Let $(A,\mu,R)$ be a Rota-Baxter $G$-associative algebra of weight $\lambda$ and $\beta \colon A \to A$ be a morphism.  Then
\[
A_\beta = (A, \mu_\beta = \beta\mu, \beta, R)
\]
is a multiplicative Rota-Baxter $G$-Hom-associative algebra of weight $\lambda$.
\end{corollary}

\begin{proof}
This is the $\alpha = \Id$ special case of Theorem \ref{thm:firsttwist}.
\end{proof}

As a converse to Corollary \ref{cor1:twist}, given a  multiplicative Rota-Baxter $G$-Hom-associative algebra $\left( A,\mu ,\alpha,R \right) $, one may ask whether it is induced by an ordinary Rota-Baxter $G$-associative  algebra $(A,\widetilde{\mu},R)$.  In other words, we ask whether there exists a morphism $\alpha$ on $(A,\widetilde{\mu},R)$
such that
$\mu=\alpha\circ\widetilde{\mu}$.

Let $(A,\mu ,\alpha)$ be a multiplicative  $G$-Hom-associative algebra. Following the observation in \cite{Gohr}, when   $\alpha$ is invertible, the twisting principle  with  $\alpha^{-1}$ leads to a $G$-associative algebra. If $\alpha$
is an algebra morphism with respect to $\widetilde{\mu}$,
then $\alpha$ is also an algebra morphism
with respect to $\mu=\alpha\circ\widetilde{\mu}$ because
$$\mu(\alpha(x),\alpha(y))=\alpha\circ\widetilde{\mu}(\alpha(x),\alpha(y))=
\alpha\circ\alpha\circ\widetilde{\mu}(x,y)=\alpha\circ\mu(x,y).$$
\noindent
If $\alpha$ is bijective, then $\alpha^{-1}$ is
also an algebra automorphism. Therefore, one may use an untwist
operation on the $G$-Hom-associative algebra to recover the
$G$-associative algebra ($\widetilde{\mu}=\alpha^{-1}\circ\mu$).

The following result is a partial converse to Corollary \ref{cor1:twist}.

\begin{proposition}
Let $(A,\mu ,\alpha,R)$ be a multiplicative  Rota-Baxter  $G$-Hom-associative  algebra of weight $\lambda$ in which $\alpha$ is invertible  and commutes with $R$.  Then
\[
A' = (A,\mu_{\alpha^{-1}}=\alpha^{-1}\circ \mu ,R)
\]
is a Rota-Baxter $G$-associative  algebra of weight $\lambda$.
\end{proposition}

\begin{proof}
The associator of $A'$ and the Hom-associator of $A$ are related as
\[
\assprime = \alpha^{-2} \circ \alphaass.
\]
The $G$-Hom-associative identity in $A$ now implies
\[
\sum_{\sigma\in G} (-1)^{\varepsilon ({\sigma})} \assprime \circ \sigma
= \alpha^{-2} \left(\sum_{\sigma\in G} (-1)^{\varepsilon ({\sigma})} \alphaass \circ \sigma\right) = 0.
\]
Therefore, the $G$-associative identity holds in $A'$.

Since  $\alpha$ commutes with $R$, so does $\alpha^{-1}$. Hence $R$ is a Rota-Baxter operator for the multiplication $\mu_{\alpha^{-1}}$.
\end{proof}

Next we construct  Rota-Baxter $G$-Hom-associative algebras involving elements of the centroid of    Rota-Baxter $G$-associative algebras. The construction of Hom-algebras using elements of the centroid was initiated in \cite{BenayadiMakhlouf} for Lie algebras.

Let   $(A, \mu, \beta)$ be a Hom-algebra.  The \emph{centroid} $Cent(A)$ of $A$ is defined as the set consisting of linear self-maps $\alpha \colon A \to A$ satisfying the condition
\[
\alpha(\mu(x, y)) = \mu(\alpha(x), y) = \mu(x, \alpha (y))
\]
for all $x,y\in A$.  Notice that  if $\alpha\in Cent (A )$, then we have
\[
\mu(\alpha^p(x),\alpha^q(y))=\alpha^{p+q}\mu(x,y)
\]
for all $p,q \geq 0$.

\begin{theorem}
Let $(A, \mu,R)$ be a Rota-Baxter $G$-associative algebra of weight $\lambda$.  Suppose $\alpha\in Cent (A)$ such that $\alpha$ and $R$ commute.  Then
\[
A(n,m) = (A,\mu_{\alpha^n} = \alpha^n\mu, \alpha^m, R)
\]
is a Rota-Baxter $G$-Hom-associative algebra of weight $\lambda$ for any $n,m \geq 0$.
\end{theorem}

\begin{proof}
The assumption that $\alpha$ lies in the centroid of $A$ implies that the associator of $A$ and the Hom-associator of $A(n,m)$ are related as
\[
\anmass
= \alpha^{2n+m} \circ \alphaass.
\]
Thus, the $G$-associative identity in $A$ implies
\[
\sum_{\sigma\in G} (-1)^{\varepsilon ({\sigma})} \anmass \circ \sigma =
\alpha^{2n+m} \left(\sum_{\sigma\in G} (-1)^{\varepsilon ({\sigma})} \alphaass \circ \sigma\right) = 0.
\]
Therefore, $A(n,m)$ is a $G$-Hom-associative algebra.

To see that the Rota-Baxter identity holds in $A(n,m)$, we can reuse the diagram \eqref{muonerb} with $\beta = \alpha^n$.  Once again the left square is commutative by the Rota-Baxter identity in $A$.  The right rectangle is commutative because $\alpha$ commutes with $R$, and hence so does $\alpha^n$.
\end{proof}

\section{From Rota-Baxter Hom-algebras to Hom-preLie algebras}
\label{sec:rbtoprelie}

In this section, we construct a functor from a full subcategory of the category of Rota-Baxter Hom-Lie-admissible (or Hom-associative) algebras to the category of left Hom-preLie algebras.

\begin{theorem}\label{HAssRBtoHpLie}
Let $( A, \cdot, \alpha,R) $ be a Rota-Baxter  Hom-Lie-admissible algebra of weight $0$ such that $\alpha$ and $R$ commute.  Define the binary operation $\ast$  on $A$ by
\begin{equation}
x\ast y=R(x)\cdot y-y\cdot R(x)=[R(x),y].
\end{equation}
Then
\[
A_1 = (A, \ast, \alpha)
\]
is a left Hom-preLie algebra.
\end{theorem}

\begin{proof}
For $x,y,z\in A$, we have:
\begin{align*}
\alpha (x)\ast (y\ast z)&= \alpha (x)\ast (R(y)\cdot z-z\cdot R(y))\\
&= R(\alpha (x))\cdot (R(y)\cdot z-z\cdot R(y))-(R(y)\cdot z -z\cdot R(y))\cdot R (\alpha (x))\\
&= \alpha (R(x))\cdot (R(y)\cdot z) -\alpha (R(x))\cdot (z\cdot R(y))\\
&\relphantom{} -(R(y)\cdot z)\cdot \alpha (R (x)) +(z\cdot R(y))\cdot \alpha (R (x)),
\end{align*}
and
\begin{align*}
(x\ast y)\ast \alpha (z)&= (R(x)\cdot y-y\cdot R(x))\ast \alpha (z)\\
\ &=R(R(x)\cdot y-y\cdot R(x))\cdot \alpha (z)- \alpha (z)\cdot R(R(x)\cdot y
-y\cdot R(x))\\
\ &= R(R(x)\cdot y)\cdot \alpha (z)-R(y\cdot R(x))\cdot \alpha (z)\\
&\relphantom{} - \alpha (z)\cdot R(R(x)\cdot y)+\alpha (z)\cdot R(y \cdot R(x)).
\end{align*}
Subtracting the above terms, switching $x$ and $y$, and then subtracting the result yield:
\begin{eqnarray*}
&&\aoneass(y,x,z) - \aoneass(x,y,z)\\
&&= \alpha (x)\ast (y\ast z)-(x\ast y)\ast \alpha (z) -\alpha (y)\ast (x\ast z)+(y\ast x)\ast \alpha (z)\\
&&= \underbrace{\alpha (R(x))\cdot (R(y)\cdot z) -\alpha (R(x))\cdot (z\cdot R(y))-(R(y)\cdot z)\cdot \alpha (R (x)) +(z\cdot R(y))\cdot \alpha (R(x))}_{[\alpha(R(x)),[R(y),z]]}\\
&&\relphantom{} -R(R(x)\cdot y)\cdot \alpha (z)+R(y\cdot R(x))\cdot \alpha (z)+ \alpha (z)\cdot R(R(x)\cdot y)-\alpha (z)\cdot R(y \cdot R(x))\\
&&\relphantom{} \underbrace{- \alpha (R(y))\cdot (R(x)\cdot z) +\alpha (R(y))\cdot (z\cdot R(x))+(R(x)\cdot z)\cdot \alpha (R (y)) -(z\cdot R(x))\cdot \alpha (R(y))}_{[\alpha(R(y)),[z,R(x)]]}\\
&&\relphantom{} +R(R(y)\cdot x)\cdot \alpha (z)-R(x\cdot R(y))\cdot \alpha (z)-\alpha (z)\cdot R(R(y)\cdot x)+\alpha (z)\cdot R(x \cdot R(y)).
\end{eqnarray*}
Gathering the $5^{th}$ and $14^{th}$, $6^{th}$ and $13^{th}$, $7^{th}$ and $16^{th}$, and $8^{th}$ and $15^{th}$ terms, and using the Rota-Baxter identity \eqref{eq:RB} with $\lambda = 0$, we obtain:
\begin{eqnarray*}
&&\aoneass(y,x,z) - \aoneass(x,y,z)\\
&&= [\alpha(R(x)),[R(y),z]] + [\alpha(R(y)),[z,R(x)]]\\
&&\relphantom{} -(R(x)\cdot R(y))\cdot \alpha (z)+(R(y)\cdot R(x))\cdot \alpha (z)+ \alpha (z)\cdot (R(x)\cdot R(y))-\alpha(z)\cdot (R(y) \cdot R(x))\\
&&= [\alpha(R(x)),[R(y),z]] + [\alpha(z),[R(x),R(y)]] + [\alpha(R(y)),[z,R(x)]]\\
&&=~ \circlearrowleft_{R(x),R(y),z}{[\alpha(R(x)),[R(y),z]]}.
\end{eqnarray*}
Hom-Lie-admissibility now implies
\begin{eqnarray*}
\aoneass(y,x,z) - \aoneass(x,y,z)=0,
\end{eqnarray*}
showing that $A_1$ is a left Hom-preLie algebra.
\end{proof}

\begin{theorem}\label{HAssRBtoHpLie2}
Let $( A, \cdot, \alpha,R) $ be a Rota-Baxter  Hom-associative algebra of weight $-1$ such that $\alpha$ and $R$ commute.  Define the operation $\ast$  on $A$ by
\begin{equation}
\begin{split}
x\ast y
&= R(x)\cdot y-y\cdot R(x)- x\cdot y\\
&= [R(x),y]- x\cdot y.
\end{split}
\end{equation}
Then
\[
A_2 = (A, \ast, \alpha)
\]
is a left Hom-preLie algebra.
\end{theorem}

\begin{proof}
For $x,y,z\in A$, we have:
\begin{align*}
\alpha (x)\ast (y\ast z)
&= R(\alpha (x))\cdot (R(y)\cdot z -z\cdot R(y)-y\cdot z)\\
&\relphantom{} - (R(y)\cdot z -z\cdot R(y)-y\cdot z)\cdot R (\alpha (x))\\
&\relphantom{} -\alpha (x)\cdot (R(y)\cdot z-z\cdot R(y)-y\cdot z),
\end{align*}
and
\begin{align*}
(x\ast y)\ast \alpha (z)
&= R(R(x)\cdot y-y\cdot R(x)-x\cdot y)\cdot \alpha (z)\\
&\relphantom{} - \alpha (z)\cdot R(R(x)\cdot y-y\cdot R(x)-x\cdot y),\\
&\relphantom{} -(R(x)\cdot y-y\cdot R(x)-x\cdot y)\cdot \alpha (z).
\end{align*}
Then using the fact that $\alpha$ and $R$ commute, and Hom-associativity we obtain:
\begin{eqnarray*}
&& \atwoass(y,x,z) - \atwoass(x,y,z)\\
&&= \alpha (x)\ast (y\ast z)-(x\ast y)\ast \alpha (z) -\alpha (y)\ast (x\ast z)+(y\ast x)\ast \alpha (z)\\
&&= \alpha (R(x))\cdot (R(y)\cdot z) +(z\cdot R(y))\cdot \alpha (R (x))\\
&& \relphantom{} -\alpha (R(y))\cdot (R(x)\cdot z) -(z\cdot R(x))\cdot \alpha (R (y))\\
&&\relphantom{} -R(R(x)\cdot y)\cdot \alpha (z)+R(y\cdot R(x))\cdot \alpha (z)\\
&&\relphantom{} +R(R(y)\cdot x)\cdot \alpha (z)-R(x\cdot R(y))\cdot \alpha (z)\\
&&\relphantom{} + \alpha (z)\cdot R(R(x)\cdot y)-\alpha (z)\cdot R(y \cdot R(x))\\
&&\relphantom{} -\alpha (z)\cdot R(R(y)\cdot x)+\alpha (z)\cdot R(x \cdot R(y))\\
&&\relphantom{} + R(x\cdot y)\cdot \alpha (z) -\alpha (z)\cdot R(x\cdot y)\\
&&\relphantom{} - R(y\cdot x)\cdot \alpha (z) + \alpha (z)\cdot R(y\cdot x).
\end{eqnarray*}
The above sum vanishes by Hom-associativity and the Rota-Baxter identity \eqref{eq:RB} with $\lambda = -1$.
\end{proof}


\section{Free Rota-Baxter $G$-Hom-associative algebras}
\label{sec:HRB1}

The purpose of this section is to give an explicit construction of the free Rota-Baxter $G$-Hom-associative algebra of weight $\lambda$ associated to a Hom-module.  Recall that $\HomRB_\lambda$ and $\HomMod$ denote the categories of Rota-Baxter $G$-Hom-associative algebras of weight $\lambda$ and of Hom-modules, respectively.  Here is the main result of this section.

\begin{theorem}
\label{thm:freehrb}
The forgetful functor
\[
\cO \colon \HomRB_\lambda \to \HomMod
\]
given by
\[
\cO(A,\mu,\alpha,R) = (A,\alpha)
\]
admits a left adjoint.
\end{theorem}

\begin{proof}
We consider an intermediate category $\bD$, which we will make precise below.  There are forgetful functors
\begin{equation}
\label{eq:O3}
\HomRB_\lambda \xrightarrow{\cO_2} \bD \xrightarrow{\cO_1} \HomMod,
\end{equation}
whose composition is $\cO$.  We will show that each of these two forgetful functors $\cO_i$ admits a left adjoint $\cF_i$ (Theorems \ref{thm:freeD} and \ref{thm:G}).  The composition
\[
\cF = \cF_2 \circ \cF_1
\]
is then the desired left adjoint in Theorem \ref{thm:freehrb}.  It remains to prove Theorems \ref{thm:freeD} and \ref{thm:G}.
\end{proof}

We begin by defining the intermediate category $\bD$.

\subsection{The category $\mathbf{D}$}
\label{subsec:D}

Let $\mathbf{D}$ be the category whose objects are quadruples $(A,\mu,\alpha,R)$ in which:
\begin{enumerate}
\item $A$ is a $\K$-module,
\item $\mu \colon A^{\otimes 2} \to A$ is a bilinear operation, and
\item $\alpha$ and $R$ are linear self-maps $A \to A$.
\end{enumerate}
A morphism $f \colon (A,\mu_A,\alpha_A,R_A) \to (B,\mu_B,\alpha_B,R_B)$ in $\mathbf{D}$ consists of a linear map $f \colon A \to B$ such that $f \circ \mu_A = \mu_B \circ f^{\otimes 2}$, $f \circ \alpha_A = \alpha_B \circ f$, and $f \circ R_A = R_B \circ f$.

There is a forgetful functor
   \begin{equation}
   \label{eq:O1}
   \cO_1 \colon \bD \to \HomMod
   \end{equation}
defined as
\[
\cO_1(A,\mu,\alpha,R) = (A,\alpha).
\]

\begin{theorem}
\label{thm:freeD}
The functor $\cO_1$ in \eqref{eq:O1} admits a left adjoint
\[
\cF_1 \colon \HomMod \to \mathbf{D}.
\]
\end{theorem}

In other words, the functor $\cF_1$ associates to a Hom-module its free object in $\mathbf{D}$.  In order to construct the free functor $\cF_1$, we need to freely generate products (for $\mu$) and images of $R$, while keeping $\alpha$ defined.  This involves an elaborate process of bookkeeping, for which we use the notion of decorated trees.  The proof of Theorem \ref{thm:freeD} will be given after the following preliminary discussions of trees.

\subsection{Planar binary trees}
\label{subsec:trees}

For $n \geq 1$, let $T_n$ denote the set of (isomorphism classes of) planar binary trees with $n$ leaves and one root.  The first few $T_n$ are depicted below.
   \[
   T_1 = \left\lbrace \tone \, \right\rbrace,~
   T_2 = \left\lbrace \ttwo \right\rbrace,~
   T_3 = \left\lbrace \tthreeone,\, \tthreetwo \right\rbrace,~
   T_4 = \left\lbrace \tfourone,\, \tfourtwo,\, \tfourthree,\, \tfourfour,\, \tfourfive \right\rbrace.
   \]
Each node represents either a leaf, which is always depicted at the top, or an internal vertex.  An element in $T_n$ will be called an \textbf{$n$-tree}.  The set of nodes ($=$ leaves and internal vertices) in a tree $\psi$ is denoted by $N(\psi)$.  The node of an $n$-tree $\psi$ that is connected to the root (the lowest point in the $n$-tree) will be denoted by $v_{low}$.  In other words, $v_{low}$ is the lowest internal vertex in $\psi$ if $n \geq 2$ and is the only leaf if $n = 1$.

Given an $n$-tree $\psi$ and an $m$-tree $\varphi$, their \textbf{grafting}
\[
\psi \vee \varphi \in T_{n+m}
\]
is the tree obtained by placing $\psi$ on the left and $\varphi$ on the right and joining their roots to form the new lowest internal vertex, which is connected to the new root.  Pictorially, we have
   \[
   \psi \vee \varphi \,=\, \psiwedgephi.
   \]
Note that grafting is a nonassociative operation.  As we will discuss below, the operation of grafting is for generating the multiplication $\mu$.

Conversely, by cutting the two upward branches from the lowest internal vertex, each $n$-tree $\psi$ with $n \geq 2$ can be uniquely represented as the grafting of two trees, say, $\psi_1 \in T_p$ and $\psi_2 \in T_q$, where $p + q = n$.  By iterating the grafting operation, one can show by a simple induction argument that every $n$-tree $(n \geq 2)$ can be obtained as an iterated grafting of $n$ copies of the $1$-tree.

\subsection{Decorated trees}
\label{subsec:decorated}

By a \textbf{decorated $n$-tree}, we mean a pair $(\psi,f)$ consisting of an $n$-tree $\psi$ and a function $f$ from the set of nodes $N(\psi)$ to the set of non-empty finite sequences of non-negative integers, satisfying the following two conditions.
   \begin{enumerate}
   \item For any $v \in N(\psi)$, if
   \begin{equation}
   \label{eq:decoration}
   f(v) = (a_k, a_{k-1}, \ldots , a_2, a_1),
   \end{equation}
   then $a_j > 0$ for all $j \geq 2$.
   \item If $v \in N(\psi)$ is a leaf and $f(v)$ is as in \eqref{eq:decoration}, then $a_1 = 0$ implies $k = 1$.
   \end{enumerate}
The set of decorated $n$-trees is denoted by $\overline{T}_n$.  For $v \in N(\psi)$, the finite sequence $f(v)$ is called the \textbf{decoration of $v$}. For example, here is a decorated $3$-tree:
   \begin{equation}
   \label{eq:dec3tree}
   \decthreetree
   \end{equation}
The decoration of each node is depicted right next to or on top of it.

As we will discuss below in \eqref{eq:opB}, decorated $n$-trees give us a way of composing $n$ elements in a Rota-Baxter $G$-Hom-associative algebra or, more generally, in an object in $\bD$.  Here let us give a heuristic explanation.  The decoration
\[
f(v) = (a_k, a_{k-1}, \ldots , a_2, a_1)
\]
of a node $v$ is for generating the operation
\[
\cdots \alpha^{a_4} R^{a_3} \alpha^{a_2} R^{a_1}.
\]
The requirement $a_j > 0$ for $j \geq 2$ in \eqref{eq:decoration} is imposed because we do not need to generate the identity operation $\alpha^0 = R^0$.  The other requirement, that $a_1 = 0$ implies $k=1$ in \eqref{eq:decoration} for a leaf, is imposed because a Hom-module $(M,\alpha)$ already has the operations $\alpha^n$ for all $n \geq 1$.  Thus, we begin by generating the operations $R^n$ for $n \geq 1$.

Given a decorated $n$-tree $(\psi,f)$ and elements $x_1, \ldots , x_n$ in an object $A$ in $\bD$, imagine placing $x_i$ at the $i$th leaf, counting from left to right, of $\psi$.  The decoration $(\ldots,a_2,a_1)$ at the $i$th leaf tells us that we should apply $\cdots R^{a_3}\alpha^{a_2}R^{a_1}$ to $x_i$.  Once this is done to each of the $x_i$, we move downward along the branches.  At each internal node, we first multiply the two elements coming from the two branches above the node.  Then we apply $\cdots R^{b_3}\alpha^{b_2}R^{b_1}$ to the resulting product according to the decoration $(\ldots, b_2, b_1)$ of that node.  Once we have done this for the lowest internal node, we have an element in $A$.

\subsection{Operations on decorated trees}
\label{subsec:op}

In order to parametrize $n$-ary operations in an object in $\bD$, we need operations on decorated trees corresponding to $\mu$, $R $, and $\alpha$, which we now discuss.

If $(\psi,f) \in \overline{T}_n$ and $(\varphi,g) \in \overline{T}_m$ are decorated trees, then their \textbf{grafting} is the $(n+m)$-tree $\psi \vee \varphi$ with decoration
   \[
   h(v) =
   \begin{cases}
   f(v) & \text{if $v \in N(\psi)$}, \\
   g(v) & \text{if $v \in N(\varphi)$}, \\
   (0) & \text{if $v = v_{low}$ in $\psi \vee \varphi$}.
   \end{cases}
   \]

For each $n \geq 1$, define a function $R \colon \overline{T}_n \to \overline{T}_n$ by setting
   \[
   R(\psi,f) = (\psi, R  f),
   \]
where $R f$ is equal to $f$, except that the decoration
   \[
   f(v_{low}) = (a_k, a_{k-1}, \ldots , a_1)
   \]
of the lowest internal vertex (or the only leaf if $\psi \in T_1$) is replaced by
   \begin{equation}
   \label{eq:betaf}
   (R f)(v_{low}) =
   \begin{cases}
   (a_k + 1, a_{k-1}, \ldots , a_1) & \text{if $k$ is odd}, \\
   (1, a_k, a_{k-1}, \ldots , a_1) & \text{if $k$ is even}
   \end{cases}
   \end{equation}

Denote by $(1,(a_1))$ the $1$-tree with decoration $a_1$ for its only leaf.  Define the functions
   \[
   \alpha \colon
   \begin{cases}
   \overline{T}_1 \diagdown \lbrace (1,(0))\rbrace \to \overline{T}_1 \diagdown \lbrace (1,(0))\rbrace, \\
   \overline{T}_n \to \overline{T}_n & \text{for $n \geq 2$}
   \end{cases}
   \]
by setting
   \[
   \alpha(\psi,f) = (\psi,\alpha f),
   \]
where $\alpha f$ is equal to $f$, except that
   \begin{equation}
   \label{eq:alphaf}
   (\alpha f)(v_{low}) =
   \begin{cases}
   (1, a_k, a_{k-1}, \ldots , a_1) & \text{if $k$ is odd}, \\
   (a_k + 1, a_{k-1}, \ldots , a_1) & \text{if $k$ is even}.
   \end{cases}
   \end{equation}
We emphasize that $\alpha(1,(0))$ is not defined.

Note that a simple induction argument shows that every decorated $n$-tree $(\psi,f)$ can be obtained uniquely from $n$ copies of $(1,(0))$ by iterating the operations $\vee$, $\alpha$, and $R$.

\subsection{Decorated trees acting on $\mathbf{D}$}
\label{subsec:opD}

Let $(B,\mu,\alpha_B,R_B)$ be an object in $\mathbf{D}$.  For each $n \geq 1$, define a map
   \begin{equation}
   \label{eq:opB}
   \gamma \colon \bbK \lbrack \overline{T}_n \rbrack \otimes B^{\otimes n} \to B
   \end{equation}
as follows.  Pick elements $\tau \in \overline{T}_n$, $\sigma \in \overline{T}_m$, and $b_1, b_2, \ldots \in B$.  Then we set
   \[
   \gamma((1,(0)); b_1) = b_1,
   \]
and, inductively,
   \[
   \begin{split}
   \gamma(\alpha\tau; b_1 \otimes \cdots \otimes b_n) &= \alpha_B(\gamma(\tau; b_1 \otimes \cdots \otimes b_n)), \\
   \gamma(R\tau; b_1 \otimes \cdots \otimes b_n) &= R_B(\gamma(\tau; b_1 \otimes \cdots \otimes b_n)), \\
   \gamma(\tau \vee \sigma; b_1 \otimes \cdots \otimes b_{n+m}) &= \mu(\gamma(\tau; b_1 \otimes \cdots \otimes b_n), \gamma(\sigma; b_{n+1} \otimes \cdots \otimes b_{n+m})).
   \end{split}
   \]
For example, for the decorated $3$-tree $\tau$ in \eqref{eq:dec3tree}, we have
   \[
   \gamma(\tau; b_1 \otimes b_2 \otimes b_3) =
   R \alpha^8\left\lbrace [\alpha^7R^4\alpha^9R^2((R^3\alpha^5R^2(b_1)) \cdot b_2)] \cdot [\alpha^2R^6(b_3)]\right\rbrace.
   \]
Here, and in what follows, we write $\mu(x,y)$ as $x \cdot y$ or even $xy$.

\begin{proof}[Proof of Theorem \ref{thm:freeD}]
Pick a Hom-module $(A,\alpha_A)$.  Consider the $\K$-module
   \begin{equation}
   \label{eq:FA}
   \cF_1(A) = \bigoplus_{n \geq 1,\, \tau \in \overline{T}_n} A^{\otimes n}_\tau,
   \end{equation}
where each $A^{\otimes n}_\tau$ is a copy of $A^{\otimes n}$.  Identify $A$ as a submodule of $\cF_1(A)$ via the inclusion
   \begin{equation}
   \label{eq:iota}
   \iota \colon A \xrightarrow{=} A_{(1,(0))} \hookrightarrow \cF_1(A).
   \end{equation}
A typical generator in $A^{\otimes n}_\tau$ is denoted by $(a_1 \otimes \cdots \otimes a_n)_\tau$.  We claim that $\cF_1$ is the desired left adjoint of $\cO_1$.  First, for $\cF_1(A)$ to be an object in $\bD$, we need to equip $\cF_1(A)$ with the three operations $\mu$, $\alpha$, and $R$.  We make use of the operations on decorated trees discussed in section \ref{subsec:op}

The bilinear operation
\[
\mu \colon \cF_1(A) \otimes \cF_1(A) \to \cF_1(A)
\]
is defined as
   \begin{equation}
   \label{eq:muF}
   \mu\left((a_1 \otimes \cdots \otimes a_n)_\tau, (a_{n+1} \otimes \cdots \otimes a_{n+m})_\sigma\right) = (a_1 \otimes \cdots \otimes a_{n+m})_{\tau \vee \sigma}
   \end{equation}
on the generators.  The linear operations $\alpha$ and $R$ on $\cF_1(A)$ are defined as
   \begin{equation}
   \label{eq:alphaF}
   \begin{split}
   \alpha\left((a_1 \otimes \cdots \otimes a_n)_\tau\right) &=
   \begin{cases}
   (a_1 \otimes \cdots \otimes a_n)_{\alpha(\tau)} & \text{if $\tau \not= (1,(0))$},\\
   \left(\alpha_A(a_1)\right)_{(1,(0))} & \text{if $\tau = (1,(0))$},
   \end{cases}\\
   R\left((a_1 \otimes \cdots \otimes a_n)_\tau\right) &= (a_1 \otimes \cdots \otimes a_n)_{R(\tau)}.
   \end{split}
   \end{equation}
With these operations, $\cF_1(A)$ becomes an object in $\mathbf{D}$.  It is clear how $\cF_1$ is defined on morphisms of Hom-modules and that $\cF_1$ defines a functor $\HomMod \to \mathbf{D}$.  Moreover, $\iota$ extends to a morphism
\[
\iota \colon A \to \cF_1(A)
\]
of Hom-modules.

To show that $\cF_1$ is the left adjoint of the forgetful functor $\cO_1$, pick an object $(B,\mu_B,\alpha_B,R_B) \in \mathbf{D}$. Let $f \colon (A,\alpha_A) \to (B,\alpha_B)$ be a morphism of Hom-modules.  We need to show that there exists a unique morphism
\[
\varphi \colon \cF_1(A) \to B \in \mathbf{D}
\]
such that
   \begin{equation}
   \label{eq:f}
   f = \varphi \circ \iota.
   \end{equation}
Define $\varphi \colon \cF_1(A) \to B$ by setting
   \begin{equation}
   \label{eq:varphi}
   \varphi\left((a_1 \otimes \cdots \otimes a_n)_\tau\right)
   = \gamma(\tau; f(a_1) \otimes \cdots \otimes f(a_n))
   \end{equation}
on the generators and extending linearly, where $\gamma$ is defined in \eqref{eq:opB}.  It is clear that \eqref{eq:f} is satisfied.

It remains to show that $\varphi$ is a morphism in $\mathbf{D}$.  Suppose that $\tau \not= (1,(0))$.  To show that $\varphi$ is compatible with $\alpha$, we compute as follows:
   \[
   \begin{split}
   \varphi\left(\alpha(a_1 \otimes \cdots \otimes a_n)_\tau\right)
   &= \varphi\left((a_1\otimes \cdots \otimes a_n)_{\alpha(\tau)}\right)\\
   &= \gamma(\alpha(\tau); f(a_1) \otimes \cdots \otimes f(a_n)) \\
   &= \alpha_B\left(\gamma(\tau; f(a_1) \otimes \cdots \otimes f(a_n))\right)\\
   &= \alpha_B\left(\varphi((a_1 \otimes \cdots \otimes a_n)_\tau)\right).
   \end{split}
   \]
Replacing $\alpha$ by $R$, the same argument, regardless of whether $\tau = (1,(0))$ or not, also shows that $\varphi$ is compatible with $R$.  Now if $\tau = (1,(0))$, then we have
   \[
   \begin{split}
   \varphi\left(\alpha((a)_{(1,(0))})\right)
   &= \varphi\left((\alpha_A(a))_{(1,(0))}\right)\\
   &= \gamma\left((1,(0)); f(\alpha_A(a))\right)\\
   &= f(\alpha_A(a))\\
   &= \alpha_B(f(a))\\
   &= \alpha_B\gamma\left((1,(0)); f(a)\right)\\
   &= \alpha_B\varphi\left((a)_{(1,(0))}\right).
   \end{split}
   \]
This shows that $\varphi$ is compatible with $\alpha$.

Likewise, to show that $\varphi$ is compatible with $\mu$, we compute as follows:
   \[
   \begin{split}
   \varphi&\left(\mu\left((a_1 \otimes \cdots \otimes a_n)_\tau, (a_{n+1} \otimes \cdots \otimes a_{n+m})_\sigma\right)\right) \\
   &= \varphi\left((a_1 \otimes \cdots \otimes a_{n+m})_{\tau \vee \sigma}\right)\\
   &= \gamma\left(\tau \vee \sigma; f(a_1) \otimes \cdots \otimes f(a_{n+m})\right) \\
   &= \mu_B\left(\gamma(\tau; f(a_1) \otimes \cdots \otimes f(a_n)), \gamma(\sigma; f(a_{n+1}) \otimes \cdots \otimes f(a_{n+m}))\right) \\
   &= \mu_B\left(\varphi((a_1 \otimes \cdots \otimes a_n)_\tau), \varphi((a_{n+1} \otimes \cdots \otimes a_{n+m})_\sigma)\right).
   \end{split}
   \]
This shows that $\varphi \colon \cF_1(A) \to B$ is a morphism in $\mathbf{D}$.

The uniqueness of $\varphi$ is clear.  Indeed, suppose that
\[
\phi \colon \cF_1(A) \to B \in \mathbf{D}
\]
is another morphism such that
\[
f = \phi \circ \iota.
\]
The definitions \eqref{eq:muF} and \eqref{eq:alphaF} and the fact that $\phi$ is compatible with $\mu$, $\alpha$, and $R$ together imply that $\phi$ is uniquely determined by $\phi((a)_{(1,(0))})$, i.e., its restriction to $A$.  But the restriction of $\phi$ to $A$ is equal to $f$, so $\phi$ must be equal to $\varphi$.
\end{proof}

To finish the proof of Theorem \ref{thm:freehrb}, consider the forgetful functor
   \begin{equation}
   \label{eq:O2}
   \cO_2 \colon \HomRB_\lambda \to \mathbf{D}
   \end{equation}
that simply forgets the extra axioms satisfied by $\mu$, $\alpha$, and $R$.  This forgetful functor is a full and faithful embedding.  In particular, we identify $\HomRB_\lambda$ as a full subcategory of $\mathbf{D}$ using $\cO_2$.

\begin{theorem}
\label{thm:G}
The category $\HomRB_\lambda$ is a reflective subcategory of $\bD$, i.e., the functor $\cO_2$ in \eqref{eq:O2} admits a left adjoint
\[
\cF_2 \colon \mathbf{D} \to \HomRB_\lambda.
\]
\end{theorem}

In order to prove Theorem \ref{thm:G}, we need to introduce ideals and quotients for objects in $\mathbf{D}$.

\subsection{Ideals and quotients in $\mathbf{D}$}
\label{subsec:ideal}

Let $(A,\mu,\alpha,R)$ be an object in $\mathbf{D}$.  An \emph{ideal in $A$} is a submodule $I \subseteq A$ that is closed under both $\alpha$ and $R$ and such that
\[
\mu(I,A) \subseteq I \andspace \mu(A,I) \subseteq I.
\]
In this case, the quotient module $A/I$ is naturally an object in $\mathbf{D}$ in which $\mu$, $\alpha$, and $R$ are induced by those in $A$.

For example, if $f \colon A \to B$ is a morphism in $\mathbf{D}$, then the kernel of $f$ is an ideal in $A$.

\begin{proposition}
\label{prop:ideal}
Let $(A,\mu,\alpha,R)$ be an object in $\mathbf{D}$ and $S$ be a non-empty subset in $A$.  Then there exists a unique ideal $\langle S \rangle$ in $A$ such that
   \begin{enumerate}
   \item $S \subseteq \langle S \rangle$, and
   \item if $I$ is an ideal in $A$ that contains $S$, then $I$ also contains $\langle S \rangle$.
   \end{enumerate}
\end{proposition}

The ideal $\langle S \rangle$ is called the \emph{ideal generated by $S$}.

\begin{proof}
Given elements $x_1, \ldots , x_n \in A$, by a \emph{parenthesized monomial}
\[
x_1 \cdots x_n \in A,
\]
we mean any possible way of multiplying the $x_i$'s using $\mu$ in the prescribed order.  For example, there is only one parenthesized monomial $x_1x_2$.  There are two parenthesized monomials $x_1x_2x_3$, namely, $(x_1x_2)x_3$ and $x_1(x_2x_3)$.

We construct an infinite increasing sequence of submodules of $A$,
   \begin{equation}
   \label{eq:sequence}
   S \subseteq S^1 \subseteq S^2 \subseteq \cdots \subseteq A,
   \end{equation}
as follows.  Define
   \[
   S^1 = span_\bbK\lbrace x_1 \cdots x_m \in A \colon m \geq 1, \text{at least one $x_j$ in $S$} \rbrace.
   \]
In other words, $S^1$ is the submodule of $A$ generated by all the parenthesized monomials in $A$ with at least one entry in $S$.  Inductively, suppose that the submodules
   \[
   S \subseteq S^1 \subseteq \cdots \subseteq S^n
   \]
have been defined for some $n \geq 1$.  Then we set
   \[
   S^{n+1} = span_\bbK\lbrace x_1 \cdots x_m \in A \colon m \geq 1, \text{at least one $x_j$ in $S^n \cup \alpha(S^n) \cup R(S^n)$}\rbrace.
   \]
In other words, $S^{n+1}$ is the submodule of $A$ generated by all the parenthesized monomials in $A$ with at least one entry in $S^n \cup \alpha(S^n) \cup R(S^n)$.

Now that we have an increasing sequence of submodules as in \eqref{eq:sequence}, we define
   \[
   \langle S \rangle = \bigcup_{n \geq 1} S^n.
   \]
It is clear that $\langle S \rangle$ is a submodule of $A$ that contains $S$.  To see that $\langle S \rangle$ is an ideal in $A$, pick elements $x \in \langle S \rangle$ and $y \in A$.  Then $x \in S^n$ for some finite $n$.  Therefore, both $\alpha(x)$ and $R(x)$ lie in $S^{n+1} \subseteq \langle S \rangle$.  Likewise, both $xy$ and $yx$ lie in $S^{n+1} \subseteq \langle S \rangle$.  This shows that $\langle S \rangle$ is an ideal in $A$ that contains $S$.

Suppose that $I$ is another ideal in $A$ that contains $S$.  Since $I$ is closed under multiplication with elements in $A$ on either side, $I$ also contains $S^1$.  Inductively, if $I$ contains $S^n$, then it also contains $\alpha(S^n)$ and $R(S^n)$ because $I$ is closed under both $\alpha$ and $R$.  It follows that $I$ contains $S^{n+1}$, again because $I$ is closed under multiplication with elements in $A$ on either side.  By induction, $I$ must contain $\langle S \rangle$ as well.

This proves the existence of the ideal $\langle S \rangle$.  The uniqueness of $\langle S \rangle$ is obvious from its two properties in the statement of the Proposition.
\end{proof}

\begin{proof}[Proof of Theorem \ref{thm:G}]
Pick an object $(A,\mu,\alpha,R) \in \mathbf{D}$.  Let $S$ be the subset of $A$ consisting of the generating relations in a typical Rota-Baxter $G$-Hom-associative algebra, i.e., the elements
\[
\sum_{\sigma\in G} (-1)^{\varepsilon ({\sigma})} \alphaass \circ \sigma (x,y,z)
\]
and
\[
R(x)R(y) - R(R(x)y + xR(y) + \lambda xy)
\]
for all $x,y,z \in A$.  Then we set
   \begin{equation}
   \label{eq:G}
   \cF_2(A) = A/\langle S \rangle,
   \end{equation}
the quotient of $A$ by the ideal $\langle S \rangle$ generated by $S$.  The quotient $\cF_2(A)$ is an object in $\mathbf{D}$ with the induced operations from $A$.  Moreover, it follows immediately from its construction that $\cF_2(A)$ is an object in the subcategory $\HomRB_\lambda$.  The construction of $\cF_2$ is clearly functorial, so we have a functor $\cF_2 \colon \mathbf{D} \to \HomRB_\lambda$.

To show that $\cF_2$ is the left adjoint of $\cO_2$, pick an object $B \in \HomRB_\lambda$.  Let $f \colon A \to B \in \mathbf{D}$ be a morphism.  We must show that there exists a unique morphism
\[
\zeta \colon \cF_2(A) \to B \in \HomRB_\lambda
\]
such that
\[
\zeta \circ p = f,
\]
where $p \colon A \to \cF_2(A)$ is the projection map.  In other words, we must show that $f$ factors through the quotient $A/\langle S \rangle$.  It suffices to show that $\langle S \rangle$ is contained in the kernel of $f$.  Since $\langle S \rangle$ is defined as the union $\cup_{n \geq 1} S^n$, it is enough to show that
\[
f(S^n) = 0
\]
for all $n \geq 1$.

It is clear that $f(S) = 0$, since $B \in \HomRB_\lambda$ and $f$ commutes with $\mu$, $\alpha$, and $R$.  This implies that $f(S^1) = 0$.  Inductively, suppose that $f(S^n) = 0$.  Then we have
\[
f(\alpha(S^n)) = \alpha(f(S^n)) = 0
\]
and, similarly, $f(R(S^n)) = 0$.  This implies that $f(S^{n+1}) = 0$.  This finishes the induction and shows that $f(S^n) = 0$ for all $n \geq 1$, as desired.
\end{proof}


\section{Hom-(tri)dendriform algebras}\label{sect4}

The purpose of this section is to study Hom-(tri)dendriform algebras.  We discuss some construction results for Hom-(tri)dendriform algebras and observe that Hom-dendriform algebras give rise to Hom-associative and Hom-preLie algebras. Free Hom-(tri)dendriform algebras are also discussed.

\subsection{Hom-dendriform algebras}
Dendriform algebras were introduced by Loday in \cite{Loday1}. Dendriform algebras are algebras with two operations, which dichotomize the notion of associative algebra.   We now generalize this structure by twisting the identities by a linear map.

\begin{definition} \label{def:HomDendr}
A \emph{Hom-dendriform algebra} is a quadruple
\[
(A, \prec,\succ, \alpha)
\]
consisting of a $\K$-module $A$ and linear maps $\prec, \succ \colon A\otimes A \rightarrow A$
and $\alpha\colon A \rightarrow A$ that satisfy the identities
\begin{eqnarray}\label{HomDendriCondition1}  (x\prec y)\prec \alpha (z)&=& \alpha(x)\prec(y\prec z +y\succ z),
 \\ \label{HomDendriCondition2} (x\succ y)\prec \alpha (z)&=&\alpha(x)\succ(y\prec z),\\
\label{HomDendriCondition3} (x\prec y+x\succ y)\succ \alpha (z)&=&\alpha(x)\succ(y\succ z)
\end{eqnarray}
for  $x, y, z$ in $A$.  A Hom-dendriform algebra is \emph{multiplicative} if $\alpha \circ \mu = \mu \circ \alpha^{\otimes 2}$ for $\mu = ~\prec$ and $\succ$.
\end{definition}

Let $(A, \prec,\succ, \alpha) $ and
$(A', \prec',\succ', \alpha') $ be two Hom-dendriform algebras. A \emph{morphism}
$f \colon A\rightarrow A'$ of Hom-dendriform algebras is a linear map such that
$$
 \prec'\circ~ (f\otimes f)=f ~\circ  \prec, \quad
   \succ'\circ~ (f\otimes f)= f ~\circ  \succ, \quad \text{
and } \quad f\circ \alpha=\alpha^{\prime }\circ f.
$$
The category of Hom-dendriform algebras is denoted by $\Homdidend$.

A dendriform algebra  is a Hom-dendriform algebra with $\alpha = \Id$.  In particular, dendriform algebras form a full-subcategory of $\Homdidend$.

The following results are the analogs of the construction results, Theorem \ref{thm:firsttwist} - Corollary \ref{cor1:twist}, for Hom-dendriform algebras.

\begin{theorem}
\label{thm:homdendtwist}
Let $(A, \prec,\succ, \alpha)$ be a Hom-dendriform algebra and $\beta \colon A \to A$ be a morphism.  Then
\[
A_\beta = (A, \prec_\beta ~= \beta ~\circ \prec, \succ_\beta ~= \beta ~\circ \succ, \beta\alpha)
\]
is also a Hom-dendriform algebra, which is multiplicative if $A$ is.

Moreover, suppose $(A', \prec', \succ', \alpha')$ is a Hom-dendriform algebra, $\beta' \colon A' \to A'$ is a morphism, and $f \colon A \to A'$ is a morphism such that $f\beta = \beta'f$.  Then
\[
f \colon A_\beta \to A'_{\beta'}
\]
is a morphism.
\end{theorem}

\begin{proof}
The Hom-dendriform axioms \eqref{HomDendriCondition1}-\eqref{HomDendriCondition3} for $A_\beta$ are obtained from those of $A$ by applying $\beta^2$.  Notice that this assertion does not use the commutativity of $\beta$ with $\alpha$.  The assertions about multiplicativity and the morphism $f$ are obvious.
\end{proof}

The following two results are special cases of Theorem \ref{thm:homdendtwist}.

\begin{corollary}
Let $(A, \prec,\succ, \alpha)$ be a multiplicative Hom-dendriform algebra.  Then
\[
A_{\alpha^n} = \left(A, \alpha^n ~\circ \prec, \alpha^n ~\circ \succ, \alpha^{n+1}\right)
\]
is also a multiplicative Hom-dendriform algebra for each $n \geq 1$.
\end{corollary}

\begin{proof}
This is the $\beta = \alpha^n$ special case of Theorem \ref{thm:homdendtwist}.
\end{proof}

\begin{corollary}\label{ theoremConstrHomDend}
Let $(A,\prec, \succ)$ be a dendriform  algebra and $\beta \colon  A\rightarrow A$ be a morphism.  Then
\[
A_\beta = (A, \beta ~\circ \prec, \beta ~\circ \succ, \beta)
\]
is a multiplicative Hom-dendriform algebra.
\end{corollary}

\begin{proof}
This is the $\alpha = \Id$ special case of Theorem \ref{thm:homdendtwist}.
\end{proof}

We now show that Hom-dendriform algebra structures lead to  Hom-associative algebra structures, hence to $G$-Hom-associative algebra structures.  We provide also  a connection to Hom-preLie algebras.

\begin{theorem}
\label{hdtoha}
Let $(A, \prec,\succ, \alpha) $   be a Hom-dendriform algebra. Define a linear map $\star\colon A\otimes A\rightarrow A$ by
\begin{equation}\label{DendToAss}
x \star y=x\prec y+ x \succ y
\end{equation}
for $x,y \in A$.  Then
\[
A_a = (A,\star, \alpha)
\]
is a Hom-associative algebra, hence also a $G$-Hom-associative algebra for every subgroup $G$ of $\mathcal{S}_3$.
\end{theorem}

\begin{proof}
To check the Hom-associative identity \eqref{HomAssCond}, observe that for $x,y,z\in A$, the sum of the left-hand sides of the Hom-dendriform axioms \eqref{HomDendriCondition1} - \eqref{HomDendriCondition3} is
\[
(x\star y)\star \alpha (z).
\]
Likewise, the sum of the right-hand sides of the axioms \eqref{HomDendriCondition1} - \eqref{HomDendriCondition3} is
\[
\alpha (x)\star (y\star z).
\]
Therefore, the sum of the three Hom-dendriform axioms says
\[
(x\star y)\star \alpha (z) = \alpha (x)\star (y\star z),
\]
which is the Hom-associative identity for $A_a$.
\end{proof}

The following result says that a Hom-dendriform algebra gives rise to a left Hom-preLie algebra via a mixed commutator.

\begin{theorem}\label{HDendtoHpLie}
Let $(A, \prec,\succ, \alpha)$ be a Hom-dendriform algebra.  Define the linear map $\rhd\colon A\otimes A\rightarrow A$ by
\[
x\rhd y=x\succ y- y \prec x
\]
for $x,y\in A$.  Then
\[
A_l = (A,\rhd, \alpha)
\]
is a left Hom-preLie algebra.
\end{theorem}

\begin{proof}
To check the left Hom-preLie identity \eqref{HomPreLieCond}, observe that for $x,y\in A$ we have:
\begin{equation}
\label{rhd1}
\begin{split}
(x\rhd y)\rhd \alpha (z)
&=(x\succ y-y\prec x)\rhd \alpha (z)\\
&= (x\succ y)\succ \alpha (z)-(y\prec x)\succ \alpha (z)\\
&\relphantom{} -\alpha (z)\prec (x\succ y)+\alpha (z)\prec (y\prec x)
\end{split}
\end{equation}
and
\begin{equation}
\label{rhd2}
\begin{split}
\alpha (x)\rhd (y\rhd z)&=\alpha (x)\rhd (y\succ z-z\prec y)\\
&= \alpha (x)\succ (y\succ z)-\alpha (x)\succ (z\prec y)\\
&\relphantom{} -(y\succ z)\prec \alpha (x)+(z\prec y)\prec \alpha (x)\\
&= (x\prec y)\succ \alpha (z)+(x\succ y)\succ \alpha (z)
-\alpha (x)\succ (z\prec y)\\
&\relphantom{} -(y\succ z)\prec \alpha (x)+
 \alpha(z)\prec (y\prec x)+\alpha(z)\prec (y\succ x).
\end{split}
\end{equation}
In the last equality above, we used the Hom-dendriform axioms \eqref{HomDendriCondition1} and \eqref{HomDendriCondition3}. Using \eqref{rhd1}, \eqref{rhd2}, and the remaining Hom-dendriform axiom \eqref{HomDendriCondition2}, the Hom-associator of $A_l$ becomes
\[
\begin{split}
\assal(x,y,z)
&= (x \rhd y) \rhd \alpha(z) - \alpha(x) \rhd (y \rhd z)\\
&= - (x \prec y + y \prec x) \succ \alpha(z)\\
&\relphantom{} - \alpha(z) \prec (x \succ y + y \succ x)\\
&\relphantom{} + (x \succ z) \prec \alpha(y) + (y \succ z) \prec \alpha(x).
\end{split}
\]
The last expression is symmetric in the variables $x$ and $y$, so $A_l$ is a left Hom-preLie algebra.
\end{proof}

The next result is the right Hom-preLie version of the previous result.

\begin{theorem}
\label{hdtorhl}
Let $(A, \prec,\succ, \alpha)$ be a Hom-dendriform algebra.  Define the linear map $\lhd\colon A\otimes A\rightarrow A$ by
\[
x\lhd y=x\prec y- y \succ x
\]
for $x,y\in A$.  Then
\[
A_r = (A,\lhd, \alpha)
\]
is a right Hom-preLie algebra.
\end{theorem}

\begin{proof}
The proof is essentially the same as that of Theorem \ref{HDendtoHpLie}, so we will omit the details.
\end{proof}

\begin{remark}
Recall that $G$-Hom-associative algebras are automatically Hom-Lie-admissible algebras (Proposition \ref{ThmAdmi}).  In particular, given a Hom-dendriform algebra $A$, the Hom-associative algebra $A_a$ (Theorem \ref{hdtoha}), the left Hom-preLie algebra $A_l$ (Theorem \ref{HDendtoHpLie}), and the right Hom-preLie algebra $A_r$ (Theorem \ref{hdtorhl}) are all Hom-Lie-admissible algebras.  In fact, direct computation shows that their commutator Hom-Lie algebras (Remark \ref{commutatorhomlie}) are equal.
\end{remark}


\subsection{ Free Hom-dendriform algebras}

Next we discuss the free Hom-dendriform algebra associated to a Hom-module.

\begin{theorem}
\label{thm:freehd}
The forgetful functor
\[
\cO \colon \Homdidend \to \HomMod
\]
given by
\[
\cO(A, \prec, \succ, \alpha) = (A,\alpha)
\]
admits a left adjoint.
\end{theorem}

\begin{proof}[Sketch of proof]
The proof is similar to that of Theorem \ref{thm:freehrb}, so we only give a sketch.  Instead of the category $\bD$ in section \ref{subsec:D}, here we use the intermediate category $\bE$ whose objects are quadruples
\[
(A,\mu_l,\mu_r,\alpha),
\]
in which $A$ is a $\bbK$-module, both $\mu_l$ and $\mu_r$ are bilinear maps on $A$, and $\alpha$ is a linear self-map on $A$.  The forgetful functor $\cO$ factors as the composition of two forgetful functors,
\[
\cO_2 \colon \Homdidend \to \bE \andspace
\cO_1 \colon \bE \to \HomMod,
\]
similar to \eqref{eq:O3}.

The left adjoint
\[
\cF_1 \colon \HomMod \to \bE
\]
of $\cO_1$ is constructed with a suitable modification of the proof of Theorem \ref{thm:freeD}.  We define a \emph{modified decorated $n$-tree} to mean a pair $(\psi,f)$, where $\psi$ is an $n$-tree, and $f$ is a function from the set of internal nodes (and not leaves) of $\psi$ to the product
\[
\bbZ_{\geq 0} \times \{l,r\}.
\]
On these modified decorated trees, the operation $\alpha$ is defined as adding $1$ to the integer component of the decoration of the lowest internal node.  There are two grafting operations $\vee_l$ and $\vee_r$, which assign the decorations $(0,l)$ and $(0,r)$, respectively, to the new lowest internal node.  Similar to \eqref{eq:opB}, each such modified decorated $n$-tree gives a way of multiplying $n$ elements in an object in $\bE$.  At an internal node with decoration $(m,*)$, where $* = l$ or $r$, one uses the multiplication $\mu_*$ to multiply the two elements from the two branches above the node.  Then one applies $\alpha^m$ to the resulting product.  One defines $\cF_1$ as in \eqref{eq:FA} using modified decorated trees and checks that it is the desired left adjoint of $\cO_1$.

For the left adjoint
\[
\cF_2 \colon \bE \to \Homdidend
\]
of $\cO_2$, one modifies the proof of Theorem \ref{thm:G}.  For an object $(A,\mu_l,\mu_r,\alpha) \in \bE$, an \emph{ideal} is a submodule $I$ that is closed under $\alpha$ and such that
\[
\mu_*(I,A) \subseteq I \andspace \mu_*(A,I) \subseteq I
\]
for $* \in \{l,r\}$.  There is an obvious analog of Proposition \ref{prop:ideal} for the existence and uniqueness of an ideal generated by a non-empty subset.  One then defines $\cF_2$ as in \eqref{eq:G}, where $S$ here is the set of generating relations in a typical Hom-dendriform algebra \eqref{HomDendriCondition1} - \eqref{HomDendriCondition3}.
\end{proof}

\subsection{Hom-tridendriform algebras}
The notion of tridendriform algebra was introduced by Loday and Ronco in \cite{Loday-Ronco04}.  A tridendriform algebra is a vector space equipped with 3 binary operations $\prec, \succ, \bullet $ satisfying seven relations. We extend this notion to the Hom situation as follows.

\begin{definition} \label{def:HomTriDendr}
A \emph{Hom-tridendriform algebra} is a quintuple
\[
(A, \prec,\succ, \bullet, \alpha)
\]
consisting of a $\K$-module $A$ and linear maps $\prec, \succ , \bullet\colon A\otimes A \rightarrow A$
and $\alpha\colon A \rightarrow A$ that satisfy the following axioms for all $x,y,z \in A$:
\begin{eqnarray}\label{HomTriDendriCondition1}   (x\prec y)\prec \alpha (z)&=& \alpha(x)\prec(y\prec z +y\succ z+ y \bullet z),
 \\ \label{HomTridDendriCondition2} (x\succ y)\prec \alpha (z)&=&\alpha(x)\succ(y\prec z),\\
\label{HomTriDendriCondition3} \alpha(x)\succ(y\succ z)&=&(x\prec y+x\succ y+x \bullet y)\succ \alpha (z),\\
\label{HomTridDendriCondition4} (x\prec y) \bullet \alpha (z)&=&\alpha(x) \bullet(y\succ z),\\
\label{HomTridDendriCondition5} (x\succ y) \bullet \alpha (z)&=&\alpha(x)\succ(y \bullet z),\\
\label{HomTridDendriCondition6} (x \bullet y)\prec \alpha (z)&=&\alpha(x) \bullet(y\prec z),\\
\label{HomTridDendriCondition7} (x \bullet y) \bullet \alpha (z)&=&\alpha(x) \bullet(y \bullet z).
\end{eqnarray}
A \emph{morphism} of Hom-tridendriform algebras is a morphism of the underlying Hom-modules that is compatible with the three binary operations.  The category of Hom-tridendriform algebras is denoted by $\Homtridend$.  A Hom-tridendriform algebra is \emph{multiplicative} if the twisting map $\alpha$ is a morphism of Hom-tridendriform algebras.  A \emph{tridendriform algebra} is a Hom-tridendriform algebra with $\alpha = \Id$.
\end{definition}

\begin{remark}
Every Hom-tridendriform algebra gives a Hom-dendriform algebra by setting $x \bullet y=0$ for any $x,y\in A$.  Also, the binary operation $\bullet$ satisfies the Hom-associativity identity \eqref{HomAssCond} by the axiom \eqref{HomTridDendriCondition7}.
\end{remark}

Some of the results above for Hom-dendriform algebras have obvious analogs for Hom-tridendriform algebras, whose proofs are slight modifications of those given above.  Therefore, we omit the proofs of the following results.

\begin{theorem}
Let $(A, \prec,\succ, \bullet, \alpha)$ be a Hom-tridendriform algebra and $\beta \colon A \to A$ be a morphism.  Then
\[
A_\beta = (A, \prec_\beta ~= \beta ~\circ \prec, \succ_\beta ~= \beta ~\circ \succ, \bullet_\beta = \beta \circ \bullet,  \beta\alpha)
\]
is also a Hom-tridendriform algebra, which is multiplicative if $A$ is.

Moreover, suppose $(A', \prec', \succ', \bullet', \alpha')$ is a Hom-tridendriform algebra, $\beta' \colon A' \to A'$ is a morphism, and $f \colon A \to A'$ is a morphism such that $f\beta = \beta'f$.  Then
\[
f \colon A_\beta \to A'_{\beta'}
\]
is a morphism.
\end{theorem}

\begin{corollary}
Let $(A, \prec,\succ, \bullet, \alpha)$ be a multiplicative Hom-tridendriform algebra.  Then
\[
A_{\alpha^n} = \left(A, \alpha^n ~\circ \prec, \alpha^n ~\circ \succ, \alpha^n \circ \bullet, \alpha^{n+1}\right)
\]
is also a multiplicative Hom-tridendriform algebra for each $n \geq 1$.
\end{corollary}

\begin{corollary}
Let $(A,\prec, \succ, \bullet)$ be a tridendriform  algebra and $\beta \colon  A\rightarrow A$ be a morphism.  Then
\[
A_\beta = (A, \beta ~\circ \prec, \beta ~\circ \succ, \beta \circ \bullet, \beta)
\]
is a multiplicative Hom-tridendriform algebra.
\end{corollary}

\begin{theorem}
\label{htdha}
Let $(A, \prec,\succ, \bullet, \alpha)$   be a Hom-tridendriform algebra.  Define a linear map $\ast\colon A\otimes A\rightarrow A$ by
\begin{equation}
x \ast y = x\prec y+ x \succ y + x\bullet y
\end{equation}
for $x,y \in A$.  Then
\[
A_a = (A, \ast, \alpha)
\]
is a Hom-associative algebra, hence also a $G$-Hom-associative algebra for every subgroup $G$ of $\mathcal{S}_3$.
\end{theorem}

\begin{theorem}
The forgetful functor
\[
\cO \colon \Homtridend \to \HomMod
\]
given by
\[
\cO(A, \prec, \succ, \bullet, \alpha) = (A,\alpha)
\]
admits a left adjoint.
\end{theorem}


\section{Rota-Baxter Hom-associative and Hom-dendriform algebras}

In this section, we discuss an adjunction between a full subcategory of the category of Rota-Baxter Hom-associative algebras and the category of Hom-(tri)dendriform algebras.

Denote by $\HomRBA_\lambda$ the full subcategory of the category of Rota-Baxter Hom-associative algebras of weight $\lambda$, consisting of $(A,\mu,\alpha,R) \in \HomRB_\lambda$ in which the twisting map commutes with the Rota-Baxter operator, i.e., $\alpha R = R\alpha$.

\begin{theorem}
\label{thm:HRBHD}
There is an adjoint pair of functors
   \begin{equation}
   \label{eq:HDD}
   U_{HD} \colon \Homdidend \rightleftarrows \HomRBA_\lambda \colon HD,
   \end{equation}
in which the right adjoint is given by
\[
HD(A,\mu,\alpha,R) = (A,\prec,\succ,\alpha) \in \Homdidend
\]
with
   \begin{equation}
   \label{eq:inducedHDD}
   x \prec y ~=~ xR(y) + \lambda xy \andspace
   x \succ y ~=~ R(x)y
   \end{equation}
for $x, y \in A$.
\end{theorem}

Recall that $xy$ denotes $\mu(x,y)$.

\begin{proof}
First we need to show that $HD(A) = (A,\prec,\succ,\alpha)$ is a Hom-dendriform algebra for any $A = (A,\mu,\alpha,R) \in \HomRBA_\lambda$.  In other words, we need to establish the three conditions \eqref{HomDendriCondition1}-\eqref{HomDendriCondition3} for $HD(A)$.

For the axiom \eqref{HomDendriCondition1}, first note that for $y$ and $z \in A$ we have
\begin{equation}
\label{lrsum}
\begin{split}
R(y \prec z + y \succ z)
&= R\left(yR(z) + \lambda yz + R(y)z\right)\\
&= R(y)R(z)
\end{split}
\end{equation}
by the Rota-Baxter identity \eqref{eq:RB}.  Using \eqref{lrsum} we compute as follows:
\[
\begin{split}
(x \prec y) \prec \alpha(z)
&= \left(x R(y) + \lambda xy\right)R(\alpha(z)) + \lambda\left(x R(y) + \lambda xy\right)\alpha(z) \\
&= (x R(y))\alpha(R(z)) + \lambda (xy)\alpha(R(z)) + \lambda(xR(y))\alpha(z) + \lambda^2(xy)\alpha(z) \\
&= \alpha(x)(R(y)R(z)) + \lambda\alpha(x)(yR(z)) + \lambda\alpha(x)(R(y)z) + \lambda^2\alpha(x)(yz) \\
&= \alpha(x)R(y \prec z + y \succ z) + \lambda\alpha(x)(y \prec z + y \succ z)\\
&= \alpha(x) \prec (y \prec z + y \succ z).
\end{split}
\]
In the second equality above, we used $\alpha R = R\alpha$.  In the third equality, we used the Hom-associativity axiom \eqref{HomAssCond}, which holds because $A \in \HomRBA_\lambda$.  This proves \eqref{HomDendriCondition1} for $HD(A)$.

For \eqref{HomDendriCondition2} we compute as follows:
\[
\begin{split}
(x \succ y) \prec \alpha(z)
&= (R(x)y)R(\alpha(z)) + \lambda(R(x)y)\alpha(z)\\
&= (R(x)y)\alpha(R(z)) + \lambda\alpha(R(x))(yz)\\
&= \alpha(R(x))(yR(z)) + R(\alpha(x))(\lambda yz)\\
&= R(\alpha(x))\left(y R(z) + \lambda yz\right)\\
&= \alpha(x) \succ (y \prec z).
\end{split}
\]
The Hom-associativity axiom \eqref{HomAssCond} was used in the second and the third equalities above, and $\alpha R = R\alpha$ was used in the second to the fourth equalities.  This proves \eqref{HomDendriCondition2} for $HD(A)$.

For \eqref{HomDendriCondition3} we compute as follows:
\[
\begin{split}
(x \prec y + x \succ y) \succ \alpha(z)
&= R(x \prec y + x \succ y)\alpha(z)\\
&= (R(x)R(y))\alpha(z)\\
&= \alpha(R(x))(R(y)z)\\
&= R(\alpha(x))(R(y)z)\\
&= \alpha(x) \succ (y \succ z).
\end{split}
\]
We used \eqref{lrsum} in the second equality, the Hom-associativity axiom \eqref{HomAssCond} in the third equality, and $\alpha R = R\alpha$ in the fourth equality.

We have shown that $HD(A)$ is a Hom-dendriform algebra.  The functoriality of $HD$ is clear.

Now we show that $HD$ admits a left adjoint
\[
U_{HD}\colon \Homdidend \to \HomRBA_\lambda.
\]
Pick a Hom-dendriform algebra $(A,\prec,\succ,\alpha)$.  Using Theorem \ref{thm:freeD}, consider the free object
\[
(\cF_1(A),\mu,\alpha,R) \in \mathbf{D}
\]
associated to the Hom-module $(A,\alpha)$.  Let $S$ be the subset of $\cF_1(A)$ consisting of:
\begin{enumerate}
\item
$im(\mu\circ (\mu \otimes \alpha - \alpha \otimes \mu))$.
\item
$im(\alpha\circ R - R\circ\alpha)$.
\item
$im(\mu\circ R^{\otimes 2} - R\circ\mu\circ(R \otimes \Id + \Id \otimes R + \lambda \Id^{\otimes 2}))$.
\item $x \prec y - (xR(y) + \lambda xy)$  for $x,y \in A$.
\item $x \succ y - R(x)y$ for $x,y \in A$.
\end{enumerate}
Here we are identifying $A$ as a submodule of $\cF_1(A)$ via the inclusion $\iota$ \eqref{eq:iota}.  Let $\langle S \rangle$ be the ideal generated by $S$ (as in Proposition \ref{prop:ideal}).

From the definition of $S$, it is straightforward to see that the quotient
   \[
   U_{HD}(A) ~=~ \cF_1(A)/\langle S \rangle,
   \]
with the induced operations of $\mu$, $\alpha$, and $R$, is an object in $\HomRBA_\lambda$.  The functoriality of $U_{HD}$ is also clear.  There is a natural map
   \[
   i \colon A \xrightarrow{\iota} \cF_1(A) \xrightarrow{pr} U_{HD}(A).
   \]
The proof that $U_{HD}$ is the left adjoint of the functor $HD$ is essentially identical to the last two paragraphs in the proof of Theorem \ref{thm:G}, using the freeness of $\cF_1(A)$ and the definitions \eqref{eq:inducedHDD}.
\end{proof}

The following result says that, under the condition that the twisting map commutes with the Rota-Baxter operator, a Rota-Baxter Hom-associative algebra can be given a new Hom-associative structure involving both $R$ and $\lambda$.

\begin{corollary}
Let $(A,\mu,\alpha,R)$ be an object in $\HomRBA_\lambda$. Define a multiplication on $A$ by
$$x\ast y=x R(y)+R(x) y+\lambda x y
$$
for $x,y \in A$.  Then
\[
A' = (A,\ast,\alpha)
\]
is a Hom-associative algebra.  Moreover, we have
$$R(x\ast y)=R(x)R(y)\quad\text{and}\quad \widetilde{R}(x\ast y)=-\widetilde{R}(x) \widetilde{R}(y)
$$
where $\widetilde{R}(x)=-\lambda x-R(x)$.
\end{corollary}

\begin{proof}
That $A'$ is a Hom-associative algebra is an immediate consequence of Theorems \ref{hdtoha} and \ref{thm:HRBHD}. The identity involving $R$ is simply the Rota-Baxter identity \eqref{eq:RB}.  The identity involving $\widetilde{R}$ can be checked by the following computation:
\[
\begin{split}
\Rtilde(x \ast y)
&= -\lambda(x \ast y) - R(x \ast y)\\
&= -\lambda xR(y) - \lambda R(x)y - \lambda^2 xy - R(x)R(y)\\
&= -(-\lambda x - R(x))(-\lambda y - R(y))\\
&= -\Rtilde(x) \Rtilde(y).
\end{split}
\]
\end{proof}

The following result is the Hom-tridendriform analog of Theorem \ref{thm:HRBHD}.

\begin{theorem}
\label{thm:Homtridend}
There is an adjoint pair of functors
   \[
   U_{HT} \colon \Homtridend \rightleftarrows \HomRBA_\lambda \colon HT,
   \]
in which $U_{HT}$ is the left adjoint.  For $(A,\mu,\alpha,R) \in \HomRBA_\lambda$, the three binary operations in the object
   \[
   HT(A) = (A,\prec,\succ,\bullet,\alpha) \in \Homtridend
   \]
are defined as
   \[
   x \prec y= xR(y), \quad x \succ y = R(x)y, \quad x \bullet y = \lambda xy
   \]
for $x,y \in A$.
\end{theorem}

\begin{proof}
The proof is very similar to that of Theorem \ref{thm:HRBHD}, so we will omit the details.
\end{proof}

\bibliographystyle{amsplain}

\begin{thebibliography}{10}


\bibitem{AizawaSaito}
Aizawa N., Sato H., \emph{ $q$-Deformation of the Virasoro algebra with central extension,}
Phys. Lett. B
\textbf{256} (1991), no. 1, 185--190.



\bibitem{aguiar1}
Aguiar M., \emph{Pre-Poison algebras}, Lett. Math. Phys. 54 (2000), 263--277.



\bibitem{aguiar2}
Aguiar M., \emph{Infinitesimal Hopf algebras}, Contemp. Math.\textbf{ 267} (2000), 1-29.


\bibitem{Aguiar3} Aguiar M., \emph{Infinitesimal bialgebras, preLie and dendriform algebras}, in: Hopf algebras in: Lect. Notes Pure Appl. Math., vol 237, Marcel Dekker, New york, 2004, 1--33.

\bibitem{am}
Aguiar M.  and Moreira W., Combinatorics of the free Baxter algebra,  Electron. J. Combin. 13 (2006), no. 1, Research Paper 17, 38 pp.




\bibitem{AEM}Ammar F., Ejbehi   Z.  and Makhlouf A., \emph{ Cohomology and Deformations of Hom-algebras, } Journal of Lie Theory \textbf{21} No. 4,  (2011) 813--836 .

\bibitem {AmmarMakhloufJA2010}Ammar F.   and Makhlouf A., \emph{Hom-Lie algebras and Hom-Lie admissible superalgebras}, J.  Algebra, Vol. \textbf{324} (7), (2010)  1513--1528.

\bibitem{An-Bai} An H. and Bai C. \emph{From Rota-Baxter algebras to preLie algebras}, e-Print
 arXiv:0711.1389v1 (2007).


 \bibitem{AM2008} Ataguema H., Makhlouf A. and Silvestrov  S., \emph{
Generalization of $n$-ary Nambu algebras and beyond}, Journal of
Mathematical Physics \textbf{50}, 1 (2009).


\bibitem{BaiBellierGuo}Bai C.,  Bellier O., Guo L. and   Ni X., \emph{Splitting of operations, Manin products and Rota-Baxter operators}, e-Print
 arXiv:1106.6080v1 (2011).

\bibitem{Baxter} Baxter G., \emph{An analytic problem whose solution follows from a simple algebraic identity}, Pacific J. Math. \textbf{10} (1960) 731--742.

\bibitem{Drinfeld}  Belavin A. A. and  Drinfeld V.G., \emph{Solutions of the classical Yang-Baxter equation for simple Lie algebras,} Funct. Anal. Appl., \textbf{16}, (1982) 159--180.

\bibitem{BenayadiMakhlouf} Benayadi S. and Makhlouf A.,  \emph{Hom-Lie algebras with symmetric invariant nondegenerate bilinear form}, e-Print
 arXiv:1009.4226 (2010)

\bibitem{cartier} Cartier P., \emph{On the structure of free Baxter algebras}, Advances in Math., \textbf{9}, (1972) 253--265.

\bibitem{Canepl2009}  Caenepeel S.,   Goyvaerts I., \emph{Monoidal Hom-Hopf algebras}, Comm. Alg. \textbf{39} (2011) 2216--2240.

\bibitem{ChaiElinPop} Chaichian M., Ellinas D., Popowicz Z., \emph{ Quantum conformal algebra with
central extension,} Phys. Lett. B \textbf{248}, no. 1-2, (1990) 95--99.

\bibitem{ChaiKuLukPopPresn} Chaichian, M.,
Isaev A. P., Lukierski J., Popowicz Z.,
Pre\v{s}najder P., \emph{ $q$-Deformations of Virasoro
algebra and conformal dimensions,} Phys. Lett. B
\textbf{262}   (1), (1991) 32--38.

\bibitem{ChaiIsKuLuk}  Chaichian M., Kulish P.,
Lukierski J., \emph{ $q$-Deformed Jacobi identity,
$q$-oscillators and $q$-deformed
infinite-dimensional algebras,} Phys. Lett. B
\textbf{237} , no. 3-4, (1990) 401--406.

\bibitem{ChaiPopPres} Chaichian, M.,
Popowicz, Z. and Pre\v{s}najder, P., \emph{  $q$-Virasoro
algebra and its relation to the $q$-deformed KdV
system,} Phys. Lett. B \textbf{249}, no. 1,
 (1990) 63--65.

\bibitem{ConnesKreimer} Connes A.  and  Kreimer D., \emph{Hopf algebras, Renormalization and Noncommutative Geometry,}
Comm. in Math. Phys., \textbf{199}, 203, (1998).

\bibitem{CurtrZachos1} Curtright T. L., Zachos C. K., \emph{  Deforming maps for quantum algebras,} Phys. Lett. B \textbf{243}, no. 3, 237--244  (1990).

\bibitem{DaskaloyannisGendefVir}
Daskaloyannis, C., \emph{ Generalized deformed Virasoro algebras}, Modern
Phys. Lett. A \textbf{7} no. 9, (1992) 809--816.



\bibitem{Dixmier2} Dixmier J. \emph{Enveloping algebras}, Graduate studies in Math \textbf{11}, AMS, 1996.


\bibitem{KEF1} Ebrahimi-Fard K. \emph{Loday-type Algebras and the Rota-Baxter Relation}, Lett. Math. Phys. \textbf{61 }(2002) 139--147.

\bibitem{KEF-Guo1} Ebrahimi-Fard K. and Guo L. \emph{Free Rota-Baxter algebras and rooted trees}, J. Algebra App. \textbf{7} (2) (2008) 1--28.

\bibitem{KEF-Guo2} \bysame \emph{   Rota-Baxter algebras and dendriform algebras}, Jour. Pure Appl. Algebra 212 (2008) 320--339.

 \bibitem{KEF-Guo_2007} \bysame \emph{Rota-Baxter Algebras in Renormalization of Perturbative Quantum Field Theory},
Fields Institute Communications, 50, (2007) 47-105.

 \bibitem{KEF-Guo-Kreimer} Ebrahimi-Fard K., Guo L. and Kreimer D.\emph{ Spitzer's Identity and the Algebraic Birkhoff Decomposition in pQFT},
J. Phys. A: Math. Gen., 37, (2004) 11037--11052.

\bibitem{KEF-Guo-Manchon2006} Ebrahimi-Fard K., Guo L. and Manchon D. \emph{Birkhoff type decompositions and the Baker-Campbell-Hausdorff recursion},
Comm. in Math. Phys., 267, no.3, (2006) 821--845.





\bibitem{KEF-Manchon_JA2009} Ebrahimi-Fard K. and Manchon D. \emph{Dendriform equations}, J.  Algebra  \textbf{322}  (2009) 4053--4079.

\bibitem{KEF-Manchon2009}\bysame \emph{Twisted Dendriform algebras and the preLie Magnus expansion}, e-Print
 arXiv:0910.2166 (2009).

\bibitem{KEF-Bondia-Patras} Ebrahimi-Fard K., Gracia-Bondia J.M. and Patras F. \emph{Rota-Baxter algebras and new combinatorial identities}, Lett. Math. Phys. \textbf{81 }(2007) 61--75.

    \bibitem{KEF-Machon-Patras} Ebrahimi-Fard K., Manchon D. and Patras F. \emph{New  identities in Dendriform algebras}, J. Algebra \textbf{320 }(2008) 708--727.

 \bibitem{Mak-ElHamd:defoHom-alter}  Elhamdadi M. and Makhlouf A.,  \emph{Deformations of Hom-Alternative and Hom-Malcev algebras,}
     e-Print
  	arXiv:1006.2499v1  (2010).

\bibitem{FregierGohrSilv} Fregier Y.,  Gohr A. and Silvestrov S., \emph{Unital algebras of Hom-associative type and surjective or injective twistings}, J. Gen. Lie Theory Appl. Vol. \textbf{3} (4), (2009) 285--295.

\bibitem{Gohr} Gohr A. , \emph{On Hom-algebras with surjective
twisting}, J. Algebra \textbf{324} (2010) 1483--1491.

\bibitem{GozeRemm} Goze M. and Remm E.,
\emph{Lie-admissible algebras and operads}, J. Algebra \textbf{273} (2004)
 129--152.


\bibitem{Guo1} Guo L., \emph{Properties of free Baxter algebras}, Adv. Math. \textbf{151} (2000), 346--375.

\bibitem{Guo2} \bysame \emph{Baxter algebras and differential algebras}, e-Print
 arXiv:math  0407180 (2004).

\bibitem{Guo3} \bysame \emph{WHAT IS a Rota-Baxter algebra}, Notice of Amer. Math. Soc. 56 (2009) 1436--1437.

\bibitem{GuoK} Guo L. and Keigher W.,  \emph{Baxter algebras and shuffle products,} Adv. Math. \textbf{150 }(2000) 117--149.

\bibitem{Guo4} \bysame  \emph{On free Baxter algebras: completions and the internal construction}, Adv. Math. \textbf{151 }(2000) 101--127.


\bibitem{HLS} Hartwig J. T., Larsson D., Silvestrov S. D., \emph{ Deformations of Lie algebras using $\sigma$-derivations}, J. of
Algebra \textbf{295}, (2006)  314-361.

\bibitem{HS} Hellstr\"{o}m L., Silvestrov S. D., \emph{  Commuting elements in $q$-deformed Heisenberg algebras}, World Scientific (2000).

\bibitem{Hu}  Hu N., \emph{  $q$-Witt algebras,
$q$-Lie algebras, $q$-holomorph structure and
representations,}  Algebra Colloq. {\bf 6} ,
no. 1, (1999) 51--70.

 \bibitem{JinLi} Jin Q. and  Li X., \emph{Hom-Lie algebra structures on semi-simple Lie algebras},
J. Algebra, Volume \textbf{319}, Issue 4, (2008) 1398--1408


\bibitem{Kassel1} Kassel, C., \emph{Cyclic homology of differential operators,
the Virasoro algebra and a $q$-analogue}, Commun. Math. Phys. 146
(1992) 343--351.




\bibitem{LS1} Larsson D., Silvestrov S. D., \emph{  Quasi-Hom-Lie algebras, Central Extensions and 2-cocycle-like
identities,} J.  Algebra \textbf{288},  (2005) 321--344.

\bibitem{LS2} \bysame \emph{  Quasi-Lie algebras,} in {\it Noncommutative
Geometry and Representation Theory in Mathematical Physics,}
Contemp. Math., \textbf{391}, Amer. Math. Soc., Providence, RI,  (2005)  241--248.

\bibitem{LS3} \bysame \emph{  Quasi-deformations of $sl_2(\mathbb{F})$
using twisted derivations}, Comm. in Algebra \textbf{35}, (2007) 4303 -- 4318.


\bibitem{Li-Hou-Bai07} Li X., Hou D. and Bai C, \emph{Rota-Baxter operators on preLie algebras}, Journal of Nonlinear Mathematical Physics, \textbf{14} (2) (2007) 269--289.


\bibitem{LiuKeQin} Liu, Ke Qin,
\emph{Characterizations of the quantum Witt algebra}, Lett. Math. Phys.
\textbf{24} , no. 4, (1992)  257--265.


\bibitem{Loday0} Loday J.L., \emph{Cup-product for Leibniz cohomology and dual Leibniz algebras}, Math . Scand.  \textbf{77} (2) (1995) 189--196.

\bibitem{Loday1} \bysame \emph{Dialgebras}, in  Lecture Notes in Math., vol. \textbf{1763}, Springer, Berlin, 2001, 7--66.


\bibitem{Loday-Ronco04} Loday J.L. and Ronco M., \emph{Trialgebras and families of polytopes}, in Homotopy Theory: relations with Algebraic Geometry , Group Cohomology, and Algebraic K-Theory, in Contemp. Math. vol \textbf{346}, Amer. Math. Soc., Providence, RI, 2004, 369--673.



\bibitem{Mak:HomAlterna2010} Makhlouf A.,  \emph{Hom-alternative algebras and Hom-Jordan algebras}, International Electronic Journal of Algebra, Volume \textbf{8} (2010) 177--190.
\bibitem{Mak:Almeria}
\bysame
\emph{ Paradigm of Nonassociative Hom-algebras and Hom-superalgebras,} Proceedings of Jordan Structures in Algebra and Analysis Meeting, Eds:  J. Carmona Tapia, A. Morales Campoy, A. M. Peralta Pereira, M. I. Ramírez Álvarez, Publishing house: Circulo Rojo (2010), 145--177.


\bibitem{MS} Makhlouf A. and Silvestrov S. D.,
\emph{Hom-algebra structures}, J. Gen. Lie Theory Appl. \textbf{2}
(2) , (2008) 51--64.

\bibitem{HomHopf} \bysame
\emph{Hom-Lie admissible Hom-coalgebras and Hom-Hopf algebras},
 Published  as Chapter 17, pp
189-206, {\rm S. Silvestrov, E. Paal, V. Abramov, A. Stolin, (Eds.),
Generalized Lie theory in Mathematics, Physics and Beyond,
Springer-Verlag, Berlin, Heidelberg, (2008).}

\bibitem{HomDeform} \bysame
\emph{Notes on Formal deformations of Hom-Associative and Hom-Lie
algebras},  Forum Mathematicum, vol. \textbf{22} (4) (2010) 715--759.

\bibitem{HomAlgHomCoalg} \bysame
\emph{Hom-Algebras and Hom-Coalgebras},  J. of Algebra and its Applications, Vol. \textbf{9}, (2010).


\bibitem{Rota} Rota G.C., \emph{Baxter algebras and combinatorial identities}, I, II, Bull. Amer. Math. Soc. \textbf{75} (1969) 325-329;  Bull. Amer. Math. Soc. \textbf{75} (1969) 330-334.

\bibitem{Rota2}  \bysame \emph{Baxter operators, an introduction},  In Gian-Carlo Rota on combinatorics, Contemp. Mathematicians, pages 504--512. Birkh¬auser Boston, Boston, MA, 1995.

\bibitem{semenov} Semenov-Tian-Shansky M. A. ,\emph{ What is a classical r-matrix?}, Funct. Ana. Appl., \textbf{17}, no.4.,
(1983) 259--272.

\bibitem{Sheng} Sheng Y., \emph{Representations of hom-Lie algebras},  e-Print
 	arXiv:1005.0140v1 (2010).


\bibitem{Yau:EnvLieAlg} Yau D.,
\emph{Enveloping algebra of Hom-Lie algebras}, J. Gen. Lie Theory Appl.
\textbf{2} (2008) 95--108.

, arXiv:0712.3515v1 (2007).

\bibitem{Yau:homology}  \bysame
\emph{ Hom-algebras and homology,}  J. Lie Theory \textbf{19} (2009) 409--421.

\bibitem{Yau:YangBaxter} \bysame
\emph{The Hom-Yang-Baxter equation, Hom-Lie algebras,
 and quasi-triangular bialgebras,} J. Phys. A \textbf{42} (2009)  165--202.

\bibitem{Yau:comodule} \bysame
 \emph{Hom-bialgebras and comodule Hom-algebras}, Int. E. J. Alg. \textbf{8} (2010) 45--64.

\bibitem{Yau:YangBaxter2} \bysame
\emph{The Hom-Yang-Baxter equation and Hom-Lie algebras},  J. Math. Phys. \textbf{52} (2011) 053502.

\bibitem{Yau:novikov} \bysame
\emph{Hom-Novikov algebras}, J. Phys. A \textbf{44} (2011) 085202.

\bibitem{Yau:ClassicYangBaxter} \bysame
 \emph{The classical Hom-Yang-Baxter equation and Hom-Lie bialgebras},  e-Print

 arXiv:0905.1890v1, (2009).

 \bibitem{Yau:MalsevAlternJordan}\bysame
\emph{ Hom-Malsev, Hom-alternative, and Hom-Jordan algebras,} e-Print
  arXiv 1002.3944 (2010).
\end{thebibliography}
\providecommand{\bysame}{\leavevmode\hbox to3em{\hrulefill}\thinspace}
\providecommand{\MR}{\relax\ifhmode\unskip\space\fi MR }
\providecommand{\MRhref}[2]{%
  \href{http://www.ams.org/mathscinet-getitem?mr=#1}{#2}
}
\providecommand{\href}[2]{#2}

\end{document}